\documentclass[10pt]{article}

\usepackage[english]{babel}							
\usepackage[utf8]{inputenc}
\usepackage{amsmath,amsfonts,amssymb,amsthm,cancel,mathtools,empheq,latexsym}
\usepackage{graphicx}
\usepackage{subfig,epsfig,tikz,float}		        
\usepackage{booktabs,multicol,multirow,tabularx,array} 
\usepackage{geometry}
\usepackage{hyperref}
\hypersetup{
    colorlinks=true,       
    linkcolor=red,          
    citecolor=blue,        
    filecolor=magenta,      
    urlcolor=cyan           
}

\newtheorem{theorem}{Theorem}
\newtheorem{lemma}{Lemma}

\newcommand{\ave}[1]{ \left\{\!\!\left\{ {#1} \right\}\!\!\right\} }
\newcommand{\jump}[1]{ \left[ \! \left[ {#1} \right] \! \right] }

\newcommand{\triple}[1]{ |\!|\!|{#1}|\!|\!| }
\title{A boundary-field formulation for elastodynamic scattering\\
}
\author{%
George C. Hsiao$^{1}$ \ and \
Tonatiuh S\'anchez-Vizuet$^{2*}$ \ and \
Wolfgang L. Wendland$^{3}$ }

\begin{document}
\date{}
\maketitle
\vspace{-0.5cm}
\begin{center}
{\footnotesize 
$^1$ University of Delaware \qquad
$^2$ The University of Arizona \qquad
$^3$ University of Stuttgart \\
E-mails: ghsiao@udel.edu / tonatiuh@math.arizona.edu / wolfgang.wendland@mathematik.uni-stuttgart.de \\
*Corresponding author\\
}
\end{center}

\bigskip
\noindent
{\small{\bf Abstract.}
 An incoming elastodynamic wave impinges on an elastic obstacle is embedded in an infinite elastic medium. The objective of the paper is to examine the subsequent elastic fields scattered by and transmitted into the elastic obstacle. By applying a boundary-field equation method, we are able to formulate a nonlocal boundary problem (NBP) in the Laplace transformed domain, using the field equations inside the obstacle and boundary integral equations in the exterior elastic medium. Existence, uniqueness and stability of the solutions to the NBP are established in Sobolev spaces for two different integral representations. The corresponding results in the time domain are obtained. The stability bounds are translated into time domain estimates that can serve as the starting point for a numerical discretization based on Convolution Quadrature.          
}

\medskip
\noindent
{\small{\bf Keywords}{:} 
Transient wave scattering, elastodynamics, time-dependent boundary integral equations, convolution quadrature.
}

\noindent
{\bf Mathematics Subject Classifications (2010)}: 35A15, 35L05, 45A05, 74J05, 74J20.

\section{Introduction}\label{sec:Intro}

Scattering problems for the elastodynamic equations have attracted considerable interest over the years. However, most of the previous work has been based on time-harmonic formulations
(see, e.g., \cite{AhHs:1975,BaVa:1976, DaKl:2000, DaRi:1993, Ku:1963, kupradze,  RiShRe:1985}, and the recent monograph \cite{TC:2020}). Considerably fewer efforts have been carried out for transient wave propagation among which analytical studies were carried out in \cite{CrRi:1968, KiSc:2008,TC:2020}, while a computational method in two dimensions based on convolution quadrature was proposed in \cite{DoSaSa2015}.   

In the articles \cite{HsSaSa:2016, HsSaWe:2015}, we were able to show for the first time that the scattering problem for fluid/solid interaction can be formulated as a non-local boundary value problem (NBP) in the Laplace transformed domain and treated numerically with the coupling of boundary elements and finite elements in space and convolution quadrature in time \cite{Lubich:1988a, Lubich:1988b, Lubich:1994}. Moreover, on the theoretical side, we showed that based on the estimates in the Laplace domain, the corresponding solutions and estimates in the time domain can be easily obtained. Subsequently this approach has been adapted successfully in various settings involving acoustic wave propagation \cite{BrSaSa:2016,HsSV:2021, HsWe:2021a}. In this paper, following Hsiao, Sayas and Weinacht \cite{HsSaWe:2015}, we apply the time-dependent boundary-field formulation to the elastodynamic  scattering and obtain an NBP in the Laplace transformed domain.  In particular, we are able to prove the existence of variational  solution of the NBP directly as in Hsiao and Wendland \cite{HsWe:2021a} (see also \cite{BLS:2015,HSVS:2019}). Moreover, following the techniques laid out by Laliena and Sayas \cite{LaSa:2009} for treating acoustic scattering problems, we are also able to establish the variational solution of NBP indirectly. 

It is worthy mentioning that our approach based on the formulation of the NBP in the Laplace transformed domain is originally motivated from the work of  
Lubich for solving time-dependent boundary integral equations of convolution type, and has been advanced by Sayas and his coworkers (see, for instance, \cite{LaSa:2009}, the monograph \cite{Sayas:2016} and references therein).  The latter has become particularly popular in recent years for treating time-dependent boundary integral equations of convolution type, which is now  known as the convolution quadrature (CQ) method. An essential feature of this CQ method is that it works directly on the data in the time domain rather than in the transformed domain, and does not depend upon the explicit construction of the time-dependent fundamental solutions of time-dependent partial differential equations under consideration. These features make this approach particularly well suited for numerical computations; the analytic results proven in Section \ref{sec:TimeDomain} can be used to establish the minimum regularity required from problem data, as well as the convergence properties and the growth of the error in time associated to a semi-discretization based on Convolution Quadrature.

The paper is organized as follows: The governing equations for elastic wave scattering are presented in Section \ref{sec:formulation} in both the time and the Laplace domains. Section \ref{sec:BIF} is devoted to the formulation of two different coupled integro-differential systems for the problem that transform the unbounded exterior problem into a non-local (but bounded) problem posed on the interface of the scatterer. The required concepts from the theory of boundary integral equations as well as from Sobolev space theory are briefly presented in this section as well. The core of the analysis is contained in Section \ref{sec:LaplaceDomainResults}, where the variational formulations for the problems introduced in the previous section are shown to be well posed. Special care is taken to keep track of the dependence of the stability and ellipticity bounds with respect to the Laplace parameter, as these bounds contain the information required for their time-domain counterparts. The main results in the time domain follow then easily from the previous analysis and are established in the final Section \ref{sec:TimeDomain}.

\textbf{A remark on notation:} Throughout the paper we will transit between the time domain (where our problem of interest is posed) to the Laplace domain (where we will carry out the analysis) and back to the time domain (where we will finally establish existence, uniqueness regularity requirements and growth of the solution). In doing so and to avoid confusion, we will stick to the convention that functions in the time domain will be denoted by lower case letters, while functions in the Laplace domain will be denoted by upper case letters. Moreover, functions defined in the volume will be denoted by Latin letters, while we will use Greek letters for functions defined in the boundary. The table below can be used as a reference to navigate the notation.

\begin{center}
\begin{tabular}{|c|c|c|}
\cline{2-3}
\multicolumn{1}{c|}{} & Volume & Boundary \\
\hline
Time domain & $\mathbf u^-,\mathbf u^+, \mathbf u^{inc}, \mathbf f$ & $\boldsymbol \lambda, \boldsymbol\varphi$\\
\hline
Laplace domain & $\mathbf U^-,\mathbf U^+, \mathbf U^{inc}, \mathbf F$ & $\boldsymbol \Lambda, \boldsymbol\varPhi$\\
\hline
\end{tabular}
\end{center}

\section{Formulation of the problem}\label{sec:formulation}
\subsection{Governing equations in the time domain}
Let $\Omega^-$ be a connected (although not necessarily \textit{simply} connected), bounded domain in $\mathbb R^3$ with Lipschitz boundary $\Gamma$; we will denote the complementary unbounded region by $\Omega^{+} := \mathbb R^{3}\setminus\overline{\Omega^-}$ and the exterior unit normal vector pointing in the direction of $\Omega^+$ by $\boldsymbol n$. A schematic of the problem geometry is shown in Figure \ref{fig:geometry}. 
\begin{figure}[h]
\center
\includegraphics[width=0.5\linewidth]{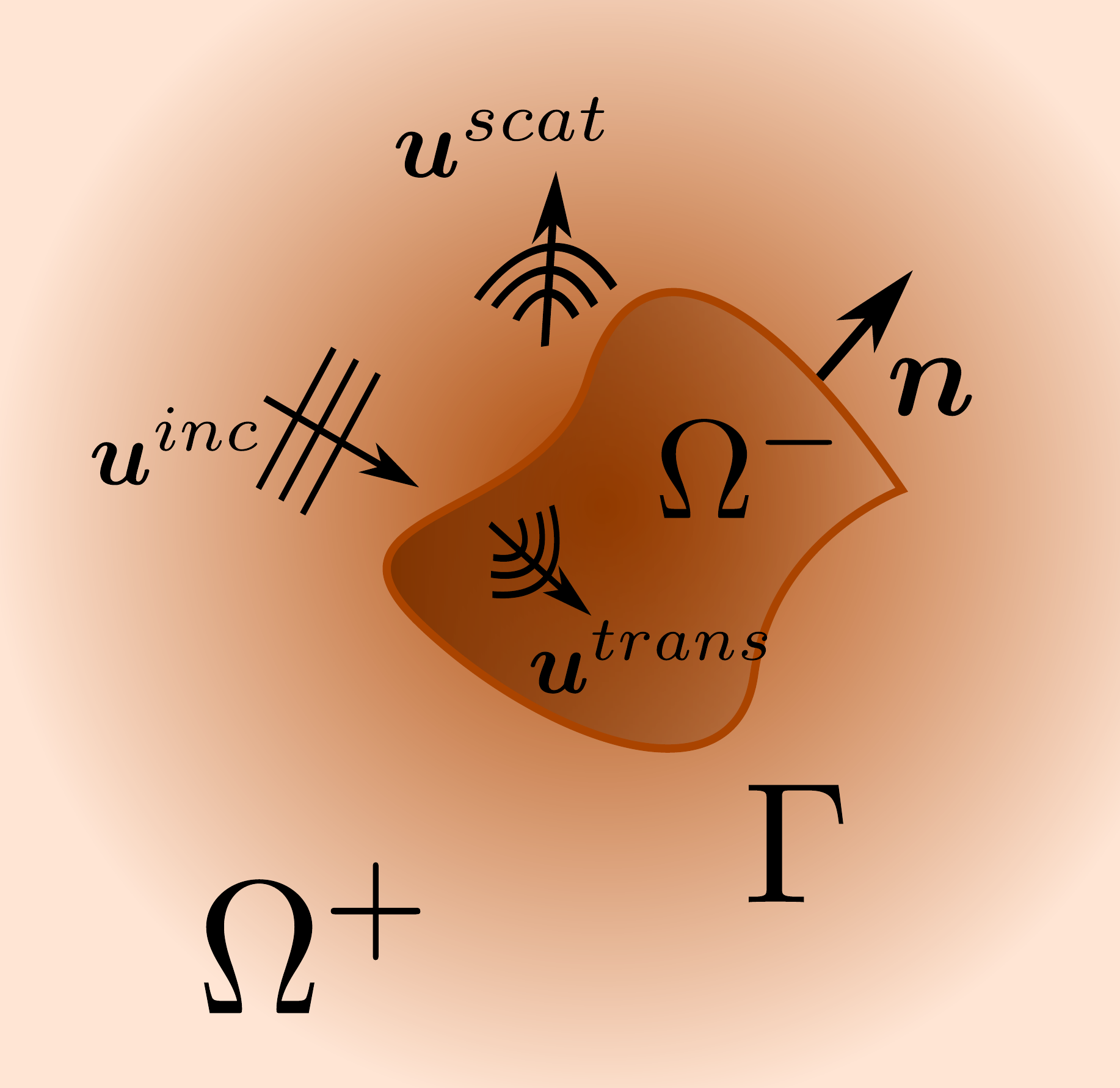}
\caption{Schematic of the problem geometry. For simplicity, the scattered and transmitted waves will be denoted in the text simply as $\mathbf u^+$ and $\mathbf u^-$ respectively.}\label{fig:geometry}
\end{figure}
Throughout this work we will consider that both $\Omega^-$ and $\Omega^+$ are occupied by linearly elastic materials. In the unbounded region $\Omega^+$ we will assume that the material is isotropic and homogeneous with density $\rho_+>0$. The latter assumption of isotropy imples that, besides its density, the material in $\Omega^+$ is characterized by the Lam\'e parameters $\lambda$ and $\mu$. These coefficients are assumed to be constants satisfying the relations
\[
\mu >0   \qquad \text{ and } \qquad K:= \frac{2\mu + 3\lambda}{3}> 0.
\]
The quantity $K$ is known as the \textit{bulk modulus} and characterizes the reaction of a given material to compressive stress, while $\mu$ (also known as the \textit{shear modulus})  characterizes the material's rigidity with respect to shear stress.

In contrast, within the bounded region $\Omega^-$ we will allow the material to be both inhomogeneous and anisotropic. More precisely, the density $\rho_-$ will be considered to be a bounded positive function almost everywhere in $\Omega^-$. In addition, we will allow for the presence of body forces acting inside $\Omega^-$ but will require that they vanish identically on $\Omega^+$. These forces will be denoted by $\mathbf f(\mathbf x,t)$ and will require that: 1) The support of $\mathbf f(\cdot,t)$ is compactly contained in $\Omega^-$ for all times, 2) $\mathbf f(\cdot,t)\in \mathbf L^2(\Omega)$ for all times, 3) $\mathbf f(\cdot,t)\equiv 0$ for all $t\leq 0$, and 4) $\mathbf f(\mathbf x,\cdot)\in\mathcal C^\infty(0,T)$ for some $T>0$. 

The isotropy/anisotropy of a material refers to the way in which mechanical stress induces deformations. Mathematically, this is reflected in the concrete way in which the \textit{linearized elastic strain tensor,}
\[
\boldsymbol \varepsilon(\mathbf u) := \frac{1}{2}\left(\nabla\mathbf u + \nabla\mathbf u^{\top}\right),
\]
relates to the \textit{elastic stress tensor}, $\boldsymbol \sigma(\mathbf u)$. According to Hooke's law, in the regime of linear elasticity and infinitesimal strain the relation is given in terms of the fourth order \textit{stiffness tensor} $\mathbf C$ by
\[
\boldsymbol\sigma (\mathbf u) : = \mathbf C \,\boldsymbol\varepsilon(\mathbf u).
\] 
The components of the stiffness tensor satisfy the following pairwise symmetry conditions
\[
\mathrm C_{ijkl} = \mathrm C_{jikl} = \mathrm C_{klij}.
\]
In addition, we will require that each of the component functions $\mathrm C_{ijkl}$ is essentially bounded in $\Omega_-$, and that there exists a constant $c_0>0$ such that
\[
\boldsymbol \varepsilon(\mathbf u) : \mathbf C \,\boldsymbol \varepsilon(\mathbf u) \geq c_0 \boldsymbol \varepsilon(\mathbf u) : \boldsymbol \varepsilon(\mathbf u),
\]
almost everywhere in $\Omega^-$.
In the isotropic region, the stress tensor can be expressed succinctly in terms of the Lam\'e parameters as 
\[
\boldsymbol\sigma_+(\mathbf u) : = \lambda\nabla\cdot\mathbf u \,\mathbf I + 2\mu\boldsymbol\varepsilon(\mathbf u),
\]
where the identity matrix has been denoted as $\mathbf I$. As customary, we will denote the elastic stress in the normal direction, or \textit{traction}, as
\[
\mathbf T^{i}\mathbf u^{i} := \boldsymbol\sigma_{i}(\mathbf u_i)\,\boldsymbol n, \qquad (i\in\{+,-\}).
\]

We will consider that a small perturbation $\mathbf u^{inc}$, supported away from $\overline{\Omega^-}$ for $t\leq 0$,  travels in the material eventually impinging on the region $\Omega^-$ for some positive time. This perturbation will induce transmitted, $\mathbf u^-$, and scattered, $\mathbf u^+$ waves that, in the regime of small stress and strain can be modeled by the equations
\begin{subequations}\label{eq:TimeDomainSystem}
\begin{alignat}{6}
-\Delta^*_+\mathbf u^+ + \rho_+\frac{\partial^2}{\partial t^2}\mathbf u^+ =\,& \boldsymbol 0  \qquad& \text{ in } \Omega^+\times(0,T), \\
-\Delta^*_-\mathbf u^- + \rho_-\frac{\partial^2}{\partial t^2}\mathbf u^- =\,& \mathbf f \qquad& \text{ in } \Omega^-\times(0,T),
\end{alignat}
which must be understood in the sense of distributions. The elastic elliptic operator $\Delta^*_i$ appearing above is defined as $\Delta^*_i\mathbf u : = \nabla\cdot\boldsymbol\sigma_i(\mathbf u)$, where the divergence operator acts along the rows of the tensor $\boldsymbol\sigma_i(\mathbf u)$. For the unbounded component, this simplifies into the familiar expression
\[
\Delta^*_+\mathbf u^+: = \mu\Delta\mathbf u^+ + (\Lambda + \mu)\nabla\,\nabla\cdot\mathbf u^+.
\]
The total displacement, along with the normal stress induced by it, must be continuous across the interface between the two materials, leading to the  transmission conditions
\begin{alignat}{6}
\mathbf u^- - \mathbf u^+ =\,& \mathbf u^{inc} \qquad& \text{ on } \Gamma\times(0,T), \\
\mathbf T^-\mathbf u^- - \mathbf T^+\mathbf u^+=\,& \mathbf T^+\mathbf u^{inc} \qquad& \text{ on } \Gamma\times(0,T).
\end{alignat}

The dynamical equations describing the evolution of the system are closed by requiring the scattered and transmitted waves to be causal, namely
\begin{alignat}{6}
\label{eq:Causality+}
\mathbf u^+ = \frac{\partial}{\partial t}\mathbf u^+ = \boldsymbol 0 \qquad \text{ on } \Omega^+\times\{0\}, \\
\label{eq:Causality-}
\mathbf u^- = \frac{\partial}{\partial t}\mathbf u^- = \boldsymbol 0 \qquad \text{ on } \Omega^-\times\{0\}.
\end{alignat}
\end{subequations}

Together, the equations comprising the system  \eqref{eq:TimeDomainSystem} fully describe the evolution of the scattered and transmitted wave.

\subsection{Governing equations in the Laplace domain}
In their seminal article \cite{BaHa:1986a} Bamberger and Ha-Duong laid the groundwork for the analysis of wave propagation problems by passing through the Laplace domain. Their goal was to study the well posedness of a boundary integral formulation of the wave equation in terms of the potentials associated to the Helmholtz operator, rather than the retarded potentials that arise directly in the time domain. Their approach, combined with Lubich's method of Convolution Quadrature (CQ) \cite{Lubich:1988a, Lubich:1988b, Lubich:1994} enables the design of numerical schemes that, taking as input time-domain data, approximate the solutions of the dynamical problem using only the potentials associated to the resolvent equation---which typically result in simpler discretizations than their time-domain counterparts. The reader interested in further details about CQ is referred to \cite{HaSa:2016} for a brief introduction to the theory and implementation details, or to the monograph \cite{Sayas:2016} for a thorough analysis and detailed applications to time domain boundary integral equations.

With the goal of formulating the problem in a way that takes advantage of Lubich's analysis technique, we will then recast the system \eqref{eq:TimeDomainSystem} in the Laplace domain; we will start by  introducing some notation. The positive complex half-plane will be denoted by
\[
\mathbb C^+ : = \{s\in\mathbb C : \mathrm{Re}\, s >0\},
\]
and for $s\in \mathbb C^+$, we will also use the notation
\[
\sigma := \mathrm{Re}\,s\, \qquad \underline{\sigma} := \min\{1,\sigma\}.
\]
We will say that a function $f$ is \textit{causal} if $f(\mathbf x,t) \equiv 0$ for all $t\leq 0$ and we will define its Laplace transform as
\[
\mathcal L\{f\}(s):= F(s) = \int_0^{\infty}f(t)e^{-st}\,dt
\]
whenever the integral converges. As laid out in \cite{LaSa:2009}, this definition can be extended naturally to Hilbert space valued distributions whenever, for any $\alpha>0$, the function $\exp(-\alpha\cdot)f$ defines a tempered distribution with values in the space of linear mappings $L(X,Y)$ between the Hilbert spaces $X$ and $Y$. In the current manuscript, we shall understand the Laplace transform in its distributional sense.

Applying the Laplace transform to the system \eqref{eq:TimeDomainSystem} presented in the previous section, we obtain the following boundary value problem in the Laplace domain that we will study for $s\in\mathbb C^+$

\begin{subequations}\label{eq:LaplaceSystem}
\begin{align}
\label{eq:ExteriorEquation}
-\Delta^*_+\mathbf U^+ + \rho_+s^2\mathbf U^+ =\,& \boldsymbol 0 \qquad&& \text{ in } \Omega^+, \\
\label{eq:InteriorEquation}
-\Delta^*_-\mathbf U^- + \rho_-s^2\mathbf U^- =\,& \mathbf F \qquad&& \text{ in } \Omega^- ,\\
\label{eq:TraceContinuity}
\mathbf U^- - \mathbf U^+ =\,& \mathbf U^{inc} \qquad&& \text{ on } \Gamma, \\
\label{eq:TractionContinuity}
\mathbf T^-\mathbf U^- - \mathbf T^+\mathbf U^+=\,& \mathbf T^+\mathbf U^{inc} \qquad&& \text{ on } \Gamma,
\end{align}
\end{subequations}
where the functions in capital letters represent the Laplace transform of their time-domain counterparts appearing in the system \eqref{eq:TimeDomainSystem}, and the causality conditions \eqref{eq:Causality+} and \eqref{eq:Causality-} are implicitly enforced by the causality of the Laplace transform. In the next couple of sections we will focus on the analysis of the Laplace-domain system above, establishing its well posedness and obtaining stability bounds in terms of the Laplace parameter $s$. These Laplace-domain results will then be transferred in the final section into time-domain statements making use of a slight improvement by Sayas \cite{Sayas:2016} of a result by Lubich \cite{Lubich:1994}.

\section{Boundary-field formulation}\label{sec:BIF}
\subsection{Sobolev space notation}
Throughout this communication we will make use of standard results on Sobolev space theory that can be found in standard references such as \cite{Adams:2003, HsWe:2021}. The space of scalar square-integrable functions defined over the open domain $\mathcal O$ will be denoted by $L^2(\mathcal O)$, while the space of scalar-valued square integrable functions over $\mathcal O$ whose distributional derivative is itself a square integrable function will be denoted by $H^1(\mathcal O)$. The vector-valued counterparts of these spaces will be denoted with boldface notation as $\boldsymbol L^2(\mathcal O):=L^2(\mathcal O)\times L^2(\mathcal O)\times L^2(\mathcal O)$ and $\boldsymbol H^1(\mathcal O):=H^1(\mathcal O)\times H^1(\mathcal O) \times H^1(\mathcal O)$ respectively. 

The $L^2$-inner product will be denoted by the symbol $(\cdot,\cdot)_{\mathcal O}$ and should be understood as
\begin{align*}
(\mathbf U,\mathbf V)_{\mathcal O}  =\,& \int_{\mathcal O} \mathbf U\cdot\mathbf V && \text{ (for vectors)},  \\
(\mathbf M,\mathbf N)_{\mathcal O}  =\,& \int_{\mathcal O} \mathbf M\,:\,\mathbf N && \text{ (for matrices)}.
\end{align*}

The $\boldsymbol L^2(\mathcal O)$ and $\boldsymbol H^1(\mathcal O)$ norms induced by this inner product will be denoted respectively as
\[
\| \mathbf U \|^2_{\mathcal O} : = ( \mathbf U,\overline{ \mathbf U})_{\mathcal O}, \qquad \qquad \| \mathbf U \|^2_{1,\mathcal O} : = ( \mathbf U,\overline{ \mathbf U})_{\mathcal O} + (\nabla \mathbf U,\overline{\nabla \mathbf  U})_{\mathcal O},
\]
where the overline denotes complex conjugation. The definition of the inner products ensures that this definition encompasses both the vector and matrix norms.

The notion of ``restriction to the boundary" can be extended to functions in $H^{1}(\Omega)$. This generalized restriction operator is known as the \textit{trace operator} and will be denoted by $\gamma$ and coincides exactly with the restriction if a function is continuous on $\overline{\mathcal O}$. The space of functions defined on $\partial\mathcal O$ and corresponding to the trace of a function in $H^1(\mathcal O)$ will be denoted as $H^{1/2}(\partial\mathcal O)$; its dual space will be denoted by $H^{-1/2}(\partial\mathcal O)$---the corresponding spaces for vector valued functions will be distinguished by the use of boldface. The duality pairing between $\mu\in H^{-1/2}(\partial\mathcal O)$ and $\eta\in H^{1/2}(\partial\mathcal O)$ will be denoted by angled brackets $\langle\mu,\eta\rangle_{\partial\mathcal O}$. This product agrees with the integral of $\int_{\partial\mathcal O}\mu\eta$ whenever the integral is well defined. The induced norms in the trace space and its dual will be denoted by $\|\cdot\|_{1/2,\partial\mathcal O}$ and $\|\cdot\|_{-1/2,\partial\mathcal O}$ respectively. In all cases, and for the sake of notational simplicity, we will omit the domain from the norm whenever it is clear from the context. 

The natural space for the solutions to the Laplace domain Navier-Lam\'e equations will be defined as
\[
\boldsymbol H^{1}_{\Delta^*}(\mathcal O) : = \left\{\mathbf U : \|\mathbf U\|_{\Delta^*,\mathcal O}<\infty \right\},
\]
where the norm in the definition is given by
\[
\|\mathbf U\|_{\Delta^*,\mathcal O}^2: = \|\mathbf U\|_{\mathcal O}^2 + \|\nabla\mathbf U\|_{\mathcal O}^2 + \|\Delta^*\mathbf U\|_{\mathcal O}^2.
\]
If the boundary of $\mathcal O$ is Lipschitz, the following integration by parts formula (Betti's formula) holds for all functions $\mathbf U \in \boldsymbol H^{1}_{\Delta^*}(\mathbf R^d)$ and $\mathbf V \in \boldsymbol H^1(\mathbb R^d)$
\begin{equation}\label{eq:BettisFormula}
\mp\langle\mathbf T^\pm\mathbf  U,\gamma^\pm\mathbf V\rangle_{\partial\mathcal O} =  (\boldsymbol\sigma(\mathbf U),\boldsymbol\varepsilon(\mathbf V))_{\mathcal O_\pm} + (\Delta^*\mathbf U,\mathbf V)_{\mathcal O_\pm},
\end{equation}
where $\mathcal O_-$ denotes the bounded region enclosed by $\partial\mathcal O$,  $\mathcal O_+:=\mathbb R^3\setminus\overline{\mathcal O_-}$ its unbounded complement, and symbols $\gamma^-$ and $\gamma^+$ denote the trace operators for the inner and outer domain respectively . Betti's formula above also serves as a definition of the distributional traction operator and will be essential in deriving the weak formulation of the problem.

We will also make use of the following energy norm for $s\in\mathbb C_+$ and $\rho>0$:
\begin{equation}\label{eq:EnergyNorm}
\triple{\mathbf U}^2_{|s|,\mathcal O} := (\boldsymbol\sigma\left(\mathbf U\right),\overline{\boldsymbol\varepsilon\left(\mathbf U\right)})_{\mathcal O} + \|s\sqrt{\rho}\,\mathbf U\|_{\mathcal O}^2. 
\end{equation}
The following norm equivalence holds
\begin{equation}\label{eq:EnergyEquivalence}
\underline{\sigma}\triple{\mathbf U}_{1,\mathcal O} \leq \triple{\mathbf U}_{|s|,\mathcal O} \leq \frac{|s|}{\underline{\sigma}}\triple{\mathbf U}_{1,\mathcal O}.
\end{equation}

Throughout the manuscript, we will make frequent use of the fact that, for $s=1$, the norm $\triple{\cdot}_{1,\mathcal O}$ is equivalent to the $\mathbf H^1(\mathcal O)$ norm. 

Finally, the symbol $\lesssim$ will be frequently used in the estimates to obviate generic constants that \textit{do not depend on the Laplace parameter $s$ or its real part $\sigma$}. Hence, the expression $\mathbf U \lesssim \mathbf V
$ should be understood as: ``There exists a positive constant $C$ independent of $s$ such that $\mathbf U \leq C \mathbf V$. The dependence of the estimates with respect to the Laplace parameter will be tracked explicitly.
\subsection{Elastic layer potentials and operators}
We will now recast the exterior part of the Laplace domain problem \eqref{eq:LaplaceSystem} in terms of boundary integral equations \cite{HsWe:2021}. To that end, we will first introduce the fundamental solution $\mathbf E(\boldsymbol x,\boldsymbol y; s)$ to the exterior three-dimensional Navier-Lam\'e equation \eqref{eq:ExteriorEquation} given by \cite[Chapter 2]{kupradze}
\[
\mathbf E(\boldsymbol x,\boldsymbol y; s) := \frac{1}{4\pi \mu}\frac{\exp(-s|\boldsymbol x-\boldsymbol y|/c_s)}{|\boldsymbol x-\boldsymbol y|}\mathbf I +\frac{1}{4\pi \rho_+s^2} \nabla\nabla^\top\left(\frac{\exp(-s|\boldsymbol x-\boldsymbol y|/c_p)-\exp(-s|\boldsymbol x-\boldsymbol y|/c_s)}{4\pi|\boldsymbol x-\boldsymbol y|} \right),
\]
where $c_s:=\sqrt{\mu/\rho_+}$ and $c_p:=\sqrt{(2\mu+\lambda)/\rho_+}$ correspond respectively to the speed of shear (transversal) and pressure (longitudinal) waves travelling through $\Omega_+$, and $\mathbf I$ denotes the identity matrix. Note that the time-harmonic fundamental solution can be recovered from the expression above by the substitution $s \mapsto -i\omega$ (see, e.g., \cite{AhHs:1975}).

Using the fundamental solution and considering $\boldsymbol\varPhi\in\boldsymbol H^{1/2}(\Gamma)$, and $\boldsymbol\Lambda\in\boldsymbol H^{-1/2}(\Gamma)$ it is possible to define two layer potentials
\begin{subequations}\label{eq:LayerPotentials}
\begin{align}
\mathcal S(s)\boldsymbol\Lambda(\boldsymbol x) :=\,& \int_\Gamma \mathbf E(\boldsymbol x,\boldsymbol y; s)\boldsymbol\Lambda(\boldsymbol y)\cdot\boldsymbol{d\Gamma}_{\boldsymbol y} \quad& \text{(Single layer potential)},\\
\mathcal D(s)\boldsymbol\varPhi(\boldsymbol x) :=\,& \int_\Gamma \mathbf T^+\mathbf E(\boldsymbol x,\boldsymbol y; s)\boldsymbol\varPhi(\boldsymbol y)\cdot\boldsymbol{d\Gamma}_{\boldsymbol y} \quad& \text{(Double layer potential)},
\end{align}
\end{subequations}
which both satisfy the equation \eqref{eq:ExteriorEquation} for any $\boldsymbol x\in \mathbb R^3\setminus\Gamma$. These functions satisfy the following conditions across the boundary $\Gamma$
\begin{equation}\label{eq:JumpConditions}
\jump{\mathcal S(s)\boldsymbol\Lambda} = \boldsymbol 0, \qquad \jump{\mathbf T\mathcal S(s)\boldsymbol\Lambda} = \boldsymbol\Lambda, \qquad \jump{\mathcal D(s)\boldsymbol\varPhi} = -\boldsymbol\varPhi, \qquad \jump{\mathbf T\mathcal D(s)\boldsymbol\varPhi} = \boldsymbol 0,
\end{equation}
where for a function $\boldsymbol v$ and all points $\boldsymbol y\in\Gamma$ the jump operator $\jump{\cdot}$ is defined as
\[
\jump{\boldsymbol v} : = \lim_{\epsilon\to 0}\left(\boldsymbol v(\boldsymbol y - \epsilon\boldsymbol n)-\boldsymbol v(\boldsymbol y + \epsilon\boldsymbol n)\right).
\]
Considering the average value of the layer potentials, rather than their jump, leads to the following four boundary integral operators
{\small \begin{equation}
\label{eq:Operators}
\begin{array}{cccc}
\mathcal V\boldsymbol\Lambda := \ave{\mathcal S(s)\boldsymbol\Lambda} \quad& \text{(Single layer)},\quad&
\mathcal K(s)\boldsymbol\varPhi :=\, \ave{\mathcal D(s)\boldsymbol\varPhi} \quad& \text{(Double layer)},\\
\mathcal K^\prime(s)\boldsymbol\Lambda :=\, \ave{\mathbf T\mathcal S(s)\boldsymbol\Lambda} \quad& \text{(Adjoint double layer)},\quad&
\mathcal W(s)\boldsymbol\varPhi :=\, -\ave{\mathbf T \mathcal D(s)\boldsymbol\varPhi} \quad& \text{(Hypersingular)},
\end{array}
\end{equation} }
where average the operator $\ave{\cdot}$ is defined as
\[
\ave{\boldsymbol v} : = \frac{1}{2}\lim_{\epsilon\to 0}\left(\boldsymbol  v(\boldsymbol y - \epsilon\boldsymbol n)+\boldsymbol  v(\boldsymbol y + \epsilon\boldsymbol n)\right) \quad \text{ for all } \boldsymbol y\in \Gamma.
\]
Combining the jump conditions \eqref{eq:JumpConditions} with the layer operators \eqref{eq:Operators} it is possible to derive the following identities:
\begin{subequations}\label{eq:TractionTraceIDs}
\begin{align}
\label{eq:TractionID}
\mathbf T^\pm \mathcal S(s)\boldsymbol\Lambda(\boldsymbol y) =\,\left(\mp\tfrac{1}{2}\mathcal I + \mathcal K^\prime(s)\right)\boldsymbol\Lambda(\boldsymbol y),\\ 
\label{eq:TraceID}
\gamma^\pm \mathcal D(s)\boldsymbol\varPhi(\boldsymbol y) =\,\left(\pm\tfrac{1}{2}\mathcal I + \mathcal K(s)\right)\boldsymbol\varPhi(\boldsymbol y),
\end{align}
\end{subequations}
where the identity operator was denoted by $\mathcal I$. 
\subsection{Two non-local boundary problems}
Since the functions defined by the single and double later potentials \eqref{eq:LayerPotentials} satisfy the exterior equation \eqref{eq:ExteriorEquation}, we will then propose two representations involving the single and double layer potentials in terms of unknown densities $\boldsymbol\Lambda$ and $\boldsymbol\varPhi$. To determine these we will make use of the transmission conditions \eqref{eq:TraceContinuity} and \eqref{eq:TractionContinuity} to derive two different boundary integral formulations. 
%
\subsubsection{A direct formulation}
%
We start by proposing an ansatz of the form
\begin{equation}\label{eq:IntegralRepresentationA}
\mathbf U^+(\boldsymbol x) = \left\{\begin{array}{cc}
\mathcal D(s)\boldsymbol\varPhi(\boldsymbol x) - \mathcal S(s)\boldsymbol\Lambda(\boldsymbol x) \qquad& \text{ for }\boldsymbol x \in\Omega+
, \\
\boldsymbol 0 \qquad& \text{ for }\boldsymbol x \in\Omega_-.
\end{array}
\right.
\end{equation}
To determine the unknown functions $\boldsymbol\Lambda$ and $\boldsymbol\varPhi$ we observe that from the transmission condition \eqref{eq:TraceContinuity} and the integral representation of $\mathbf U^+$ it follows that
\begin{alignat*}{6}
\gamma^-\mathbf U^- - \gamma^+\mathbf U^+ =\,& \gamma^-\mathbf U^- +\gamma^+\left(\mathcal S(s)\boldsymbol\Lambda - \mathcal D(s)\boldsymbol\varPhi\right) && \\
=\,& \gamma^-\mathbf U^- + \mathcal V\boldsymbol\Lambda -(\tfrac{1}{2}\mathcal I +\mathcal K(s))\boldsymbol\varPhi =\,&& \gamma^+\mathbf U^{inc},
\end{alignat*}
where we made use of the continuity of the single layer potential \eqref{eq:JumpConditions} and the identity \eqref{eq:TraceID}.

On the other hand, the ansatz for $\mathbf U^+$ for $\boldsymbol x\in \Omega_-$ implies that $\mathbf T^-\mathbf U^+=0$. From this observation two important facts follow. First, using the jump conditions for the normal traction \eqref{eq:JumpConditions}, we have that $\boldsymbol\Lambda = -\jump{\mathbf T\mathbf U^+} = \mathbf T^+\mathbf U^+$ and therefore the transmission condition \eqref{eq:TractionContinuity} can be written as
\[
\mathbf T^-\mathbf U^- -\boldsymbol\Lambda = \mathbf T^+\mathbf U^{inc}.
\]
Second, the integral representation \eqref{eq:IntegralRepresentationA}, the identity \eqref{eq:TractionID}, and $\mathbf T^-\mathbf U^+=0$ lead to
\[
\left(\tfrac{1}{2}\mathcal I + \mathcal K^\prime(s)\right)\boldsymbol\Lambda + \mathcal W(s)\boldsymbol\varPhi = \boldsymbol 0.
\]
We note that if in \eqref{eq:IntegralRepresentationA} $\mathbf U^+(\boldsymbol x)$ is only defined for $\boldsymbol x \in\Omega+$ we will arrive at the same equation 
from 
\[
\boldsymbol\Lambda = - \mathcal W(s)\boldsymbol\varPhi + \left(\tfrac{1}{2}\mathcal I - \mathcal K^\prime(s)\right)\boldsymbol\Lambda,
\]
since $\mathbf T^+\mathbf U^+ = \boldsymbol\Lambda$. Then, the above integral equation implies that 
$\mathbf T^-\mathbf U^+= 0$ and by the uniqueness of the corresponding Neumann  boundary value problem in $\Omega_-$, without loss of generality, we can conclude that $\mathbf U^+ = \boldsymbol 0$ in $\Omega_-$ as in \eqref{eq:IntegralRepresentationA}.

These last three equations together with \eqref{eq:InteriorEquation} give rise to the following non-local problem for the unknown functions $(\mathbf U^-,\boldsymbol\Lambda,\boldsymbol\varPhi)\in \boldsymbol H^1(\Omega^-)\times \boldsymbol H^{-1/2}(\Gamma)\times \boldsymbol H^{1/2}(\Gamma)$:
\begin{subequations}\label{eq:StrongNonLocal}
\begin{alignat}{6}
\label{eq:StrongNonLocalA}
-\Delta^*_-\mathbf U^- + \rho_-s^2\mathbf U^- =\,& \mathbf F \qquad&& \text{ in } \Omega^- ,\\
\label{eq:StrongNonLocalB}
\mathbf T^-\mathbf U^- -\boldsymbol\Lambda =\,& \mathbf T^+\mathbf U^{inc} \qquad&& \text{ on } \Gamma,\\
\label{eq:StrongNonLocalC}
\gamma^-\mathbf U^- +\mathcal V\boldsymbol\Lambda -(\tfrac{1}{2}\mathcal I +\mathcal K(s))\boldsymbol\varPhi =\,& \gamma^+\mathbf U^{inc}, \qquad&& \text{ on } \Gamma ,\\
\label{eq:StrongNonLocalD}
\left(\tfrac{1}{2}\mathcal I + \mathcal K^\prime(s)\right)\boldsymbol\Lambda + \mathcal W(s)\boldsymbol\varPhi =\,& \boldsymbol 0 \qquad&& \text{ on } \Gamma.
\end{alignat}
\end{subequations}

It is clear that, given a solution triplet $(\mathbf U^-,\boldsymbol\Lambda,\boldsymbol\varPhi)$ to \eqref{eq:StrongNonLocal}, we can use the densities $\boldsymbol\Lambda$ and $\boldsymbol\varPhi$ to define $\mathbf U^+$ through the integral representation \eqref{eq:IntegralRepresentationA} and the pair $(\mathbf U^-,\mathbf U^+)$ is then a solution to \eqref{eq:LaplaceSystem}. Conversely, if $(\mathbf U^-,\mathbf U^+)\in \boldsymbol H^{1}(\Omega_-)\times \boldsymbol H^{1}(\Omega_+)$ satisfies the system \eqref{eq:LaplaceSystem} we can then define
\[
\boldsymbol\Lambda := \mathbf T^+\mathbf U^+, \qquad \text{ and } \qquad \boldsymbol\varPhi := \gamma^+\mathbf U^+,
\]
and the triplet $(\mathbf U^-,\boldsymbol\Lambda,\boldsymbol\varPhi)$ will satisfy the integro-differential system \eqref{eq:StrongNonLocal}. Hence, problems \eqref{eq:LaplaceSystem} and \eqref{eq:StrongNonLocal} are equivalent.
%
\subsubsection{An alternative formulation}
Alternatively, we can propose a solution of the form
\begin{equation}\label{eq:IntegralRepresentationB}
\mathbf U^+(\boldsymbol x) = \left\{\begin{array}{cc}
s^{-1}\mathcal D(s)\boldsymbol\varPhi(\boldsymbol x) - \mathcal S(s)\boldsymbol\Lambda(\boldsymbol x) \qquad& \text{ for }\boldsymbol x \in\Omega+
, \\
\boldsymbol 0 \qquad& \text{ for }\boldsymbol x \in\Omega_-.
\end{array}
\right.
\end{equation}
With this ansatz we have that $\jump{\mathbf T\mathbf U^+} =  -\boldsymbol\Lambda = -\mathbf T^+\mathbf U^+$, and therefore the transmission condition \eqref{eq:TractionContinuity} becomes
\[
\mathbf T^-\mathbf U^- - \boldsymbol\Lambda =\, \mathbf T^+\mathbf U^{inc}.
\]
Just like before, computing $\gamma^+\mathbf U$ and $\mathbf T^-\mathbf U^+$ from the integral representation \eqref{eq:IntegralRepresentationB} with the aid of \eqref{eq:Operators} and \eqref{eq:TractionTraceIDs}, together with \eqref{eq:TraceContinuity} and the fact that $\mathbf T^-\mathbf U^+=0$ leads to a set of two boundary integral equations for $\boldsymbol\Lambda$ and $\boldsymbol\varPhi$. Putting it all together we arrive at the following non-local problem
\begin{subequations}\label{eq:StrongNonLocal2}
\begin{alignat}{6}
\label{eq:StrongNonLocal2A}
-\Delta^*_-\mathbf U^- + \rho_-s^2\mathbf U^- =\,& \mathbf F \qquad&& \text{ in } \Omega^- ,\\
\mathbf T^-\mathbf U^- -\boldsymbol\Lambda =\,& \mathbf T^+\mathbf U^{inc} \qquad&& \text{ on } \Gamma,\\
\label{eq:StrongNonLocal2B}
\gamma^-\mathbf U^- +\mathcal V\boldsymbol\Lambda -s^{-1}(\tfrac{1}{2}\mathcal I +\mathcal K(s))\boldsymbol\varPhi =\,& \gamma^+\mathbf U^{inc},\qquad&& \text{ on } \Gamma ,\\ 
\label{eq:StrongNonLocal2C} 
\left(\tfrac{1}{2}\mathcal I + \mathcal K^\prime(s)\right)\boldsymbol\Lambda + s^{-1}\mathcal W(s)\boldsymbol\varPhi =\,& \boldsymbol 0 \qquad&& \text{ on } \Gamma.
\end{alignat}
\end{subequations}

An argument analogous to the one used for the previous formulation shows that given a solution pair $(\mathbf U^-,\mathbf U^+)$ to \eqref{eq:StrongNonLocal}, then the triplet $(\mathbf U^-,\boldsymbol\Lambda:=\mathbf T^+\mathbf U^+,\boldsymbol\varPhi:=s^{-1}\gamma^+\mathbf U^+)$ satisfies \eqref{eq:StrongNonLocal2}. Reciprocally, given $(\mathbf U^-,\boldsymbol\Lambda,\boldsymbol\varPhi)$ satisfying \eqref{eq:StrongNonLocal2}, then defining $\mathbf U^+$ through \eqref{eq:IntegralRepresentationB}, the pair $(\mathbf U^-,\mathbf U^+)$ will determine a solution to \eqref{eq:StrongNonLocal}.

We remark that this relatively unusual representation of $\mathbf U^+$ is motivated from the approach used for acoustic scattering problems in \cite{HSVS:2019,HsWe:2021a}. As will be seen, the special scaling of the trace, $\gamma^+ \mathbf U^+ = s^{-1} \varPhi$, leads to the direct ellipticity of the boundary integral operators involved (compare Lemma \ref{lem:ElliticityB0} below for the alternative representation and Lemma \ref{lem:StronglyEllipticA} for the direct representation).
%
\section{Laplace-domain bounds}\label{sec:LaplaceDomainResults}
\subsection{Variational formulation of the direct problem}
Testing equation \eqref{eq:StrongNonLocalA} with $\mathbf V\in \boldsymbol H^1(\Omega^-)$, equation \eqref{eq:StrongNonLocalC} with $\boldsymbol\mu\in\boldsymbol H^{-1/2}(\Gamma)$ and equation \eqref{eq:StrongNonLocalD} with $\boldsymbol\eta\in\boldsymbol H^{1/2}(\Gamma)$, together with Betti's formula \eqref{eq:BettisFormula} and the transmission condition \eqref{eq:StrongNonLocalB} leads to the variational formulation
\begin{subequations}\label{eq:VariationalForm}
\begin{alignat}{6}
\label{eq:VariationalFormA}
(\boldsymbol\sigma(\mathbf U^-),\boldsymbol\epsilon(\mathbf V))_{\Omega_-}\!\! + s^2(\rho_-\mathbf U^-,\mathbf V)_{\Omega_-}\!\! - \langle\boldsymbol\Lambda,\gamma^-\mathbf  V\rangle_\Gamma =\,& (\mathbf F,\mathbf V)_{\Omega_-}\!\!  + \langle\mathbf T^+\mathbf U^{inc},\gamma^-\mathbf  V\rangle_\Gamma, \\
\label{eq:VariationalFormB}
\langle\boldsymbol\mu,\gamma^-\mathbf U^-\rangle_\Gamma + \langle\boldsymbol\mu,\mathcal V(s)\boldsymbol\Lambda\rangle_\Gamma -\langle\boldsymbol\mu,(\tfrac{1}{2}\mathcal I +\mathcal K(s))\boldsymbol\varPhi\rangle_\Gamma =\,& \langle\boldsymbol\mu,\gamma^+\mathbf U^{inc}\rangle_\Gamma, \\
\label{eq:VariationalFormC}
\langle\left(\tfrac{1}{2}\mathcal I + \mathcal K^\prime(s)\right)\boldsymbol\Lambda,\boldsymbol\eta\rangle_\Gamma + \langle\mathcal W(s)\boldsymbol\varPhi,\boldsymbol\eta\rangle_\Gamma =\,& \boldsymbol 0.
\end{alignat}
\end{subequations}
The key for proving well posedness of this system (laid out in \cite{LaSa:2009} for the Laplace resolvent equation) is to realize that a function defined through layer potentials---in particular through the first line of \eqref{eq:IntegralRepresentationA}---in fact satisfies a transmission problem in the bigger domain $\mathbb R^3\setminus\Gamma$. We thus have the following

\begin{lemma}
The variational problem \eqref{eq:VariationalForm}  is equivalent to  that of finding $(\mathbf U^-,\mathbf U^*)\in \boldsymbol H^1(\Omega^-)\times\boldsymbol H^1(\mathbb R^3\setminus\Gamma)$ such that 
\begin{subequations}
\label{eq:VariationalTransmission}
{\small
\begin{equation}
\label{eq:VariationalTransmissionA}
\gamma^-\mathbf U^- + \gamma^+\mathbf U^* =\,\gamma^+\mathbf U^{inc}
\end{equation}
}
and
{\small
\begin{equation}
\label{eq:VariationalTransmissionB}
(\boldsymbol\sigma(\mathbf U^*),\boldsymbol\epsilon(\mathbf W))_{\mathbb R^3\setminus\Gamma}\! + \! s^2(\rho_+\mathbf U^*\!\!,\mathbf W)_{\mathbb R^3\setminus\Gamma} \!+\! 
(\boldsymbol\sigma(\mathbf U^-),\boldsymbol\epsilon(\mathbf V))_{\Omega_-} \!\!+\! s^2(\rho_-\mathbf U^-\!\!,\mathbf V)_{\Omega_-} \!\!=\, (\mathbf F,\mathbf V)_{\Omega_-}\!\!  + \langle\mathbf T^+\mathbf U^{inc}\!\!,\gamma^-\mathbf  V\rangle_\Gamma
\end{equation}
}
\end{subequations}
for every
\[
(\mathbf V,\mathbf W) \in \mathbf H_0 := \left\{ (\mathbf V,\mathbf W) \in \boldsymbol H^1(\Omega^-)\times\boldsymbol H^1(\mathbb R^3\setminus\Gamma): \gamma^-\mathbf V + \gamma^+\mathbf W = \boldsymbol 0\right\}.
\]
\end{lemma}
\begin{proof}
Consider a solution $(\mathbf U^-,\boldsymbol\Lambda,\boldsymbol\varPhi)$ to \eqref{eq:VariationalForm} and define
\begin{equation}\label{eq:IR}
\mathbf U^* : = \mathcal S(s)\boldsymbol\Lambda - \mathcal D(s)\boldsymbol\varPhi,  \quad \text{ for } x \in \mathbb{R}^3 \setminus \Gamma.
\end{equation}
From this integral representation we can make use of \eqref{eq:Operators}, \eqref{eq:TractionID}, and \eqref{eq:TraceID} to compute
\begin{align*}
\gamma^+\mathbf U^* =\,& \mathcal V(s)\boldsymbol\Lambda\ -(\tfrac{1}{2}\mathcal I +\mathcal K(s))\boldsymbol\varPhi,\\
\mathbf T^-\mathbf U^* =\,& \left(\tfrac{1}{2}\mathcal I + \mathcal K^\prime(s)\right)\boldsymbol\Lambda +  \mathcal W(s)\boldsymbol\varPhi.
\end{align*}
Substituting the first of these expressions into \eqref{eq:VariationalFormB} leads to equation \eqref{eq:VariationalTransmissionA}, whereas the second one combined with \eqref{eq:VariationalFormC} implies that, for any $\boldsymbol\eta\in\boldsymbol H^{1/2}(\Gamma)$
\begin{equation}
\label{eq:ZeroInnerTraction}
\langle\mathbf T^-\mathbf U^*,\boldsymbol\eta\rangle_\Gamma = 0.
\end{equation}

Now, since \eqref{eq:VariationalFormA} must hold  for any $\mathbf V\in\boldsymbol H^1(\Omega_-)$, it follows that
\begin{equation}
\label{eqaux5}
(\boldsymbol\sigma(\mathbf U^-),\boldsymbol\epsilon(\mathbf V))_{\Omega_-}\!\! + s^2(\rho_-\mathbf U^-,\mathbf V)_{\Omega_-} - \langle\jump{\mathbf T\mathbf U^*},\gamma^-\mathbf V \rangle_{\Gamma}   = (\mathbf F,\mathbf V)_{\Omega_-}+ \langle \mathbf T^+\mathbf U^{inc},\gamma^-\mathbf V \rangle_{\Gamma},
\end{equation}
here we used the fact that, by construction, $\jump{\mathbf T \mathbf U^*} = \boldsymbol\Lambda$. Moreover, from the definition of $\mathbf U^*$ in terms of layer potentials it also follows that
\begin{equation}
\label{eq:AllSpace}
-\Delta^*_+\mathbf U^* +s^2\rho_+\mathbf U^* = \boldsymbol 0 \qquad \text{ in } \mathbb R^3\setminus \Gamma.
\end{equation}
Testing the equation above with $\mathbf W\in \mathbf H^1(\mathbb R^3\setminus\Gamma)$ leads to
\begin{alignat}{6}
\nonumber
(\boldsymbol\sigma(\mathbf U^*),\boldsymbol\epsilon(\mathbf W))_{\mathbb R^3\setminus\Gamma} + s^2(\rho_+\mathbf U^*\!\!,\mathbf W)_{\mathbb R^3\setminus\Gamma} =\,& \;\phantom{-} \langle\mathbf T^-\mathbf U^*,\gamma^-\mathbf W\rangle_\Gamma - \langle\mathbf T^+\mathbf U^*,\gamma^+\mathbf W\rangle_\Gamma \quad&& {\scriptsize (\text{Applying } \eqref{eq:BettisFormula})}\\
\nonumber
 =\,&  \phantom{-}\;\langle\jump{\mathbf T\mathbf U^*},\gamma^+\mathbf
W\rangle_\Gamma -\langle\mathbf T^-\mathbf U^*,\jump{\gamma \mathbf 
 W}\rangle_\Gamma  \quad&& \\
\label{eq:aux3}
=\,&  \phantom{-}\;\langle\jump{\mathbf T\mathbf U^*},\gamma^+\mathbf W\rangle_\Gamma \quad&& {\scriptsize (\text{From }\eqref{eq:ZeroInnerTraction})}.\end{alignat}
From here, it is easy to see that adding \eqref{eqaux5} and \eqref{eq:aux3} leads to the variational equation \eqref{eq:VariationalTransmissionB}, since the test functions are such that $\gamma^+\mathbf W = -\gamma^-\mathbf V$ .

Conversely, to show (15a) and (15b) imply (14), we begin 
 \eqref{eq:VariationalTransmissionB} with $(\boldsymbol 0, \mathbf W,)\in\mathbf H_0$ and compute
\begin{align*}
0 = (\boldsymbol\sigma(\mathbf U^*),\boldsymbol\epsilon(\mathbf W))_{\mathbb R^3\setminus\Gamma} +  s^2(\rho_+\mathbf U^*,\mathbf W)_{\mathbb R^3\setminus\Gamma} =\,& -\langle\mathbf T^-\mathbf U^*,\gamma^-\mathbf W\rangle_\Gamma + \langle\mathbf T^+\mathbf U^*,\gamma^+\mathbf W\rangle_\Gamma \\
& -  (\Delta_+^* \mathbf U^*,\mathbf W)_{\mathbb R^3\setminus\Gamma} +  s^2(\rho_+\mathbf U^*,\mathbf W)_{\mathbb R^3\setminus\Gamma} \\
=\,& -\langle\mathbf T^-\mathbf U^*,\gamma^-\mathbf W\rangle_\Gamma \\
& -  (\Delta_+^* \mathbf U^*,\mathbf W)_{\mathbb R^3\setminus\Gamma} +  s^2(\rho_+\mathbf U^*,\mathbf W)_{\mathbb R^3\setminus\Gamma},
\end{align*}
where in the last step we used the fact that, since the test pair belongs to $\mathbf H_0$, then $\mathbf V =\boldsymbol 0$ forces $\mathbf\gamma^+\mathbf W= \boldsymbol 0$. The equality above must also hold for all $\mathbf W$ supported away from the boundary, hence we have that
\begin{align}
\nonumber
-  (\Delta_+^* \mathbf U^*,\mathbf W)_{\mathbb R^3\setminus\Gamma} +  s^2(\rho_+\mathbf U^*,\mathbf W)_{\mathbb R^3\setminus\Gamma} =\,& 0, \\
\label{eq:aux6}
\langle\mathbf T^-\mathbf U^*,\gamma^-\mathbf W\rangle_\Gamma =\,& 0, 
\end{align}
must both hold independently. From this, we can conclude that in the sense of distributions
\begin{alignat*}{6}
- \Delta_+^* \mathbf U^* +  s^2 \rho_+\mathbf U^* =\,& 0\quad &&\mbox{in} \quad  \mathbb R^3\setminus\Gamma \\
\mathbf T^-\mathbf U^* =\,& \, 0,  \quad &&\text{on} \quad \Gamma,
\end{alignat*}
and therefore $\mathbf U^*$ admits an integral representation of the form \eqref{eq:IR}, where we will define $\boldsymbol\Lambda := \jump{\mathbf T\mathbf U^*}$ and $ \boldsymbol\varPhi := \jump{\gamma\mathbf U^*}$. If we compute $\gamma^+\mathbf U^*$ from the integral representation using \eqref{eq:Operators} and \eqref{eq:TraceID} and substitute the resulting expression into \eqref{eq:VariationalTransmissionA} we obtain \eqref{eq:VariationalFormB}. Similarly, computing $\mathbf T^-\mathbf U^*$ from the integral representation using \eqref{eq:Operators} and \eqref{eq:TractionID}, and recalling that the mapping $\gamma^-: \boldsymbol H^1(\mathbf \Omega_-) \to \boldsymbol H^{1/2}(\Gamma)$ is surjective, it follows that \eqref{eq:aux6} implies \eqref{eq:VariationalFormC}. Note that \eqref{eq:aux6} also implies that 
\begin{equation}
\label{eq:aux7}
\langle\boldsymbol\Lambda,\boldsymbol\eta\rangle_\Gamma = \langle\jump{\mathbf T\mathbf U^*},\boldsymbol\eta\rangle_\Gamma = -\langle\mathbf T^+\mathbf U^*,\boldsymbol\eta\rangle_\Gamma \qquad \forall\,\boldsymbol\eta\in\boldsymbol H^{1/2}(\Gamma).
\end{equation}

Finally, to obtain \eqref{eq:VariationalFormA} we go back to \eqref{eq:VariationalTransmissionB} and apply \eqref{eq:BettisFormula} to the terms involving $\mathbf W$, keeping in mind \eqref{eq:aux6}. This leads to 
\[
 \langle\mathbf T^+\mathbf U^*,\gamma^+\mathbf W\rangle_\Gamma + (\boldsymbol\sigma(\mathbf U^-),\boldsymbol\epsilon(\mathbf V))_{\Omega_-} \!\!+\! s^2(\rho_-\mathbf U^-\!\!,\mathbf V)_{\Omega_-} \!\!=\, (\mathbf F,\mathbf V)_{\Omega_-}\!\!  + \langle\mathbf T^+\mathbf U^{inc}\!\!,\gamma^-\mathbf  V\rangle_\Gamma.
\]
However, recalling that $\gamma^+\mathbf W = -\gamma^-\mathbf V$ and using \eqref{eq:aux7} we see that the equality above is in fact \eqref{eq:VariationalFormA}, which concludes the proof.
\end{proof}
 
The last lemma implies that in order to show the unique solvability of the non-local system \eqref{eq:VariationalForm}, it is enough to show that the transmission problem \eqref{eq:VariationalTransmission} is well posed. This motivates the definition
\[
\mathbf U_0^*:= \mathbf U^* - \mathbf U^{inc} \in \mathbf H^1_0(\mathbb R^3\setminus\Gamma),
\]
where $\mathbf U^{inc} \in \mathbf H^1(\mathbb R^3\setminus\Gamma)$ is an extension of $\gamma^+\mathbf U^{inc}$ such that $\mathbf U^{inc} \equiv \mathbf 0$ in $\Omega_-$, and of the bilinear and linear forms 
\begin{alignat*}{6}
A\left((\mathbf U^-,\mathbf U_0^*),(\mathbf V,\mathbf W)\right) :=\,&\phantom{+}\;\, (\boldsymbol\sigma(\mathbf U^-),\boldsymbol\epsilon(\mathbf V))_{\Omega_-} + s^2(\rho_-\mathbf U^-,\mathbf V)_{\Omega_-}\\
& +  (\boldsymbol\sigma(\mathbf U_0^*),\boldsymbol\epsilon(\mathbf W))_{\mathbb R^3\setminus\Gamma} + s^2(\rho_-\mathbf U_0^*,\mathbf W)_{\mathbb R^3\setminus\Gamma},  \\[1ex]
L\left((\mathbf V,\mathbf W)\right) :=\,& \phantom{-} (\mathbf F,\mathbf V)_{\Omega_-}  +  \langle\mathbf T^+\mathbf U^{inc},\gamma^-\mathbf  V\rangle_\Gamma  - A\left((\mathbf 0,\mathbf U^{inc} ),(\mathbf V,\mathbf W)\right).
\end{alignat*}
In the last term above and in the sequel, the function $\mathbf U^{inc}$ should be understood as an $\boldsymbol H^1$ extension of the boundary data into $\Omega_+$. With this notation, and defining $\mathbf U^*_0:=\mathbf U^* - \mathbf U^{inc}$, the variational problem \eqref{eq:VariationalTransmission} can posed as that of finding $(\mathbf U^-, \mathbf U^*_0) \in \mathbf H_0$ such that\\
\begin{equation}
\label{eq:CompactForm}
A\left((\mathbf U^-,\mathbf U^*_0),(\mathbf V,\mathbf W)\right) = L\left((\mathbf V,\mathbf W)\right) \qquad  \forall (\mathbf V,\mathbf W)\in\mathbf H_0.
\end{equation}

We will now show that this problem has indeed a unique solution
\begin{lemma}
\label{lem:StronglyEllipticA}
The bilinear form $A(\cdot,\cdot)$ defined above is strongly elliptic.
\end{lemma}
\begin{proof}
This follows directly from the computation
\begin{align*}
\mathrm{Re}\,\left[ \overline{s}A((\mathbf U^-,\mathbf U_0^*),(\overline{\mathbf U^-},\overline{\mathbf U_0^*}))\right] =\,& \phantom{+}\; \mathrm{Re}\,\left[\overline{s}(\boldsymbol\sigma(\mathbf U^-),\boldsymbol\epsilon(\overline{\mathbf U^-}))_{\Omega_-} + s(\rho_-s\mathbf U^-,\overline{s\mathbf U^-})_{\Omega_-}\right] \\
& + \mathrm{Re}\,\left[\overline{s}(\boldsymbol\sigma(\mathbf U_0^*),\boldsymbol\epsilon(\overline{\mathbf U_0^*}))_{\mathbb R^3\setminus\Gamma} + s(\rho_+s\mathbf U_0^*,\overline{s\mathbf U_0^*})_{\mathbb R^3\setminus\Gamma}\right]\\
=\,& \sigma\left(\triple{\mathbf U^-}_{|s|,\Omega_-}^2 + \triple{\mathbf U_0^*}_{|s|,\mathbb R^3\setminus\Gamma}^2\right). 
\end{align*}
\end{proof}
 
We will now make use of the lemma above to estimate the norms of the densities $\boldsymbol\varPhi$ and $\boldsymbol\Lambda$ in terms of the energy norm of the displacement field $\mathbf U$. In order to do so we will also need to make use of the following two lemmas.

\begin{lemma}\label{lem:NormalTraceBound}
If for a Lipschitz domain $\mathcal O$ and $s\in\mathbb C_+$ the function $\mathbf U\in \boldsymbol H_{\Delta^*}^1(\mathcal O)$ satisfies 
\[
-\Delta^*\mathbf U +s^2\mathbf U= \boldsymbol 0 \quad \text{ in } \mathcal O,
\]
then there exists $C_{\partial\mathcal O}>0$  such that
\[
\|\mathbf T\mathbf U\|_{-1/2,\partial\mathcal O}^2\leq C_{\partial\mathcal O} \frac{|s|}{\underline{\sigma}} \triple{\mathbf U}_{|s|,\mathcal O}^2.
\]
\end{lemma}

This result is a straightforward generalization of \cite[Proposition 2.5.2]{Sayas:2016}, which in turn relies on a result by Bamberger and Ha-Duong, \cite{BaHa:1986a}, regarding the existence of an optimal lifting of Dirichlet traces for the Laplace resolvent equation which can also be extented to the Navier-Lam\'e resolvent equation. Since the generalizations are straightforward, we will not show the proof of this result here and instead refer the reader to \cite{BaHa:1986a} and \cite{Sayas:2016} for the proofs of the Laplace counterparts.

The second lemma that we will need establishes the explicit dependence of the continuity of the single and double layer potentials with respect to the Laplace parameter $s$ and its real part $\sigma>0$.

\begin{lemma}
\label{lem:ContinuityBounds}
The following bounds for the single and double layer potential hold
\begin{align}
\label{eq:SContinuityBoundH1}
\|\mathcal S(s)\boldsymbol\Lambda\|_{1,\mathbb R^3\setminus\Gamma}\leq\,&C\frac{|s|}{\sigma\underline{\sigma}^2} \|\boldsymbol\Lambda\|_{-1/2} \\
\label{eq:SContinuityBoundEnergy}
\triple{\mathcal S(s)\boldsymbol\Lambda}_{|s|,\mathbb R^3\setminus\Gamma}\leq\,&C \frac{|s|}{\sigma\underline{\sigma}} \|\boldsymbol\Lambda\|_{-1/2} \\
\label{eq:DContinuityBoundH1}
\|\mathcal D(s)\boldsymbol\varPhi\|_{1,\mathbb R^3\setminus\Gamma}\leq\,&C\frac{|s|^{3/2}}{\sigma\underline{\sigma}^{3/2}} \|\boldsymbol\varPhi\|_{1/2} \\
\label{eq:DContinuityBoundEnergy}
\triple{\mathcal D(s)\boldsymbol\varPhi}_{|s|,\mathbb R^3\setminus\Gamma}\leq\,&C \frac{|s|^{3/2}}{\sigma\underline{\sigma}^{1/2}} \|\boldsymbol\varPhi\|_{1/2}
\end{align}
\end{lemma}
\begin{proof}
We start by considering $\boldsymbol\Lambda\in H^{-1/2}(\Gamma)$ and $\boldsymbol\varPhi\in H^{1/2}(\Gamma)$, and defining
\[
\mathbf U_{\mathcal S} : = \mathcal S(s)\boldsymbol\Lambda \qquad \text{ and } \qquad \mathbf U_{\mathcal D}:=\mathcal D(s)\boldsymbol\varPhi.
\]
It then follows from an argument analogous to the one employed in Lemma \ref{lem:StronglyEllipticA} that
{\small\[
\sigma\triple{\mathbf U_{\mathcal S}}_{|s|,\mathbb R^3\setminus\Gamma}^2 \leq |s|| \langle\overline{\boldsymbol\Lambda},\gamma\mathbf U_{\mathcal S}\rangle_\Gamma|\leq C|s|\|\boldsymbol\Lambda\|_{-1/2}\|\mathbf U_{\mathcal S}\|_{1,\mathbb R^3\setminus\Gamma}.
\]}

Making use the norm equivalence relations \eqref{eq:EnergyEquivalence}, the estimates \eqref{eq:SContinuityBoundH1} and \eqref{eq:SContinuityBoundEnergy} can be extracted from the above sequence of inequalities. Analogously it is possible to see that
{\small\begin{align*}
\sigma\triple{\mathbf U_{\mathcal D}}_{|s|,\mathbb R^3\setminus\Gamma}^2 \leq |s| |\langle\mathbf T \mathbf U_{\mathcal D},\overline{\boldsymbol\varPhi}\rangle_\Gamma|\leq C|s|\|\mathbf T \mathbf U_{\mathcal D}\|_{-1/2}\|\boldsymbol\varPhi\|_{1/2}\leq\,& C\frac{|s|^{3/2}}{\underline{\sigma}^{1/2}}\triple{\mathbf U_{\mathcal D}}_{|s|,\mathbb R^3\setminus\Gamma}  \|\boldsymbol\varPhi\|_{1/2},
\end{align*}}
from which \eqref{eq:DContinuityBoundEnergy} follows. One further application of \eqref{eq:EnergyEquivalence} on the left end of the sequence above leads to \eqref{eq:DContinuityBoundH1}.
\end{proof}
  
\begin{theorem}
Problem \eqref{eq:VariationalForm} is uniquely solvable and there exist constants $C_1>0$ and $C_2>0$ such that
\begin{align}
\label{eq:est1}
\left(\triple{\mathbf U^-}_{|s|,\Omega_-}^2 + \triple{\mathbf U^*}_{|s|,\mathbb R^3\setminus\Gamma}^2\right)^{1/2} \leq\,& C_1 \frac{|s|^2}{ \sigma\underline{\sigma}^2}\left(\|\mathbf F\|_{\Omega_-}^2\!\! + \|\mathbf T^+  \mathbf U^{inc}\|^2_{-1/2} + \|\gamma^+ \mathbf U^{inc}\|^2_{1/2}\right)^{1/2},\\
\label{eq:est2}
\left(\triple{\mathbf U^-}_{|s|,\Omega_-}^2\!\!\! +\|\boldsymbol\varPhi\|_{1/2,\Gamma}^2  + \|\boldsymbol\Lambda\|_{-1/2,\Gamma}^2\right)^{1/2} \leq\,& C_2 \frac{|s|^{5/2}}{\sigma\underline{\sigma}^{7/2}}\left(\|\mathbf F\|_{\Omega_-}^2\!\! + \|\mathbf T^+  \mathbf U^{inc}\|^2_{-1/2} + \|\gamma^+ \mathbf U^{inc}\|^2_{1/2}\right)^{1/2}.
\end{align}
\end{theorem}

\begin{proof}
We start by noticing that, since $\mathbf U_0^*:=\mathbf U^*-\mathbf U^{inc}$, the pair $(\mathbf U_0^-,\mathbf U_0^* )$ is such that 
\[
\gamma^-\mathbf U^- + \gamma^+\mathbf U_0^* = \gamma^-\mathbf U^- + \gamma^+\left(\mathbf U^* -\mathbf U^{inc}\right)=\,\boldsymbol 0,
\]
and therefore it satisfies problem \eqref{eq:CompactForm}. Hence, from Lemma \ref{lem:StronglyEllipticA}, and fact that
\[
|A\left((\mathbf 0,\mathbf U^{inc} ),(\mathbf U^-,\mathbf U_0^*)\right)|
\leq \triple{\mathbf U^{inc}}_{|s|,\mathbb R^3\setminus\Gamma}\,
\triple{\mathbf U^*_0}_{|s|,\mathbb R^3\setminus\Gamma}
\leq C \frac{|s|}{\underline{\sigma}}\, \|\gamma^+\mathbf U^{inc} \|_{1/2} \triple{\mathbf U^*_0}_{|s|,\mathbb R^3\setminus\Gamma},
\]
we see that
{\small \begin{alignat*}{6}
\sigma \Big(\!\triple{\mathbf U^-}_{|s|,\Omega_-}^2\!\!\! + \triple{\mathbf U_0^*}_{|s|,\mathbb R^3\setminus\Gamma}^2\!\Big)\!\leq\,& |s|\! \left(\!
(\mathbf F,\mathbf U^-)_{\Omega_-}\!\!  + \langle\mathbf T^+\mathbf U^{inc}\!\!,\gamma^-\mathbf  U^-\rangle_\Gamma 
- A\big((\mathbf 0,\mathbf U^{inc} ),(\overline{\mathbf U^-},\overline{\mathbf U_0^*})\big)\right)\\
\leq & C\frac{|s|^2}{\underline{\sigma}^2}
\!\left(\!\|\mathbf F\|_{\Omega_-}^2\!\!\! +\! \|\mathbf T^+  \mathbf U^{inc}\|^2_{-1/2}\! +\! \|\gamma^+ \mathbf U^{inc}\|^2_{1/2}\!\right)^{1/2}\!\!\left(\!\triple{\mathbf U^-}_{|s|,\Omega_-}^2\!\!\! +\! \triple{\mathbf U_0^*}_{|s|,\mathbb R^3\setminus\Gamma}^2\!\right)^{1/2}, 
\end{alignat*} }
from which we obtain 
\begin{equation}
\label{eq:aux001}
 \left(\triple{\mathbf U^-}_{|s|,\Omega_-}^2 + \triple{\mathbf U_0^*}_{|s|,\mathbb R^3\setminus\Gamma}^2\right)^{1/2} \leq  \frac{|s|^2}{\sigma\, \underline{\sigma}^2}
\left(\|\mathbf F\|_{\Omega_-}^2\!\! + \|\mathbf T^+ \mathbf U^{inc}\|^2_{-1/2} + \|\gamma^+ \mathbf U^{inc}\|^2_{1/2}\right)^{1/2} . 
\end{equation}
Then, for the solution $(\mathbf U^-, \mathbf U^*)$, we prove

{\small \begin{align*}
\triple{\mathbf U^-}_{|s|,\Omega_-}^2 \!\!\! + \triple{\mathbf U^*}_{|s|,\mathbb R^3\setminus\Gamma}^2 =\,& \triple{\mathbf U^-}_{|s|,\Omega_-}^2 \!\!\! + \triple{ \mathbf U^*_0 + \mathbf U^{inc}}_{|s|,\mathbb R^3\setminus\Gamma}^2  && \\
\lesssim\,& \triple{\mathbf U^-}_{|s|,\Omega_-}^2 \!\!\! + \triple{\mathbf U^*_0}_{|s|,\mathbb R^3\setminus\Gamma}^2 + \triple{\mathbf U^{inc}}_{|s|,\mathbb R^3\setminus\Gamma}^2 && \\
\lesssim\,& \triple{\mathbf U^-}_{|s|,\Omega_-}^2 \!\!\! + \triple{\mathbf U^*_0}_{|s|,\mathbb R^3\setminus\Gamma}^2 + \frac{|s|^2}{\underline{\sigma}^2}\|\mathbf U^{inc}\|_{1,\Omega_+}^2 && \text{(By \eqref{eq:EnergyEquivalence})}\\
\lesssim\,& \triple{\mathbf U^-}_{|s|,\Omega_-}^2 \!\!\!+ \triple{\mathbf U^*_0}_{|s|,\mathbb R^3\setminus\Gamma}^2 + \frac{|s|^2}{\underline{\sigma}^2}\|\gamma^+\mathbf U^{inc}\|_{1/2}^2 && \\
\lesssim\,& \frac{|s|^4}{\sigma^2\underline{\sigma}^4}\left(\|\mathbf F\|_{\Omega_-}^2 \!\!\! + \|\mathbf T^+\mathbf U^{inc}\|_{-1/2}^2 + \|\gamma^+\mathbf U^{inc}\|_{1/2}^2\right) + \frac{|s|^2}{\underline{\sigma}^2}\|\gamma^+\mathbf U^{inc}\|_{1/2}^2  && \text{(By \eqref{eq:aux001})} \\
\lesssim\,& \frac{|s|^4}{\sigma^2\underline{\sigma}^4}\left(\|\mathbf F\|_{\Omega_-}^2 \!\!\! + \|\mathbf T^+\mathbf U^{inc}\|_{-1/2}^2 + \|\gamma^+\mathbf U^{inc}\|_{1/2}^2\right), &&
\end{align*} }
which leads to \eqref{eq:est1}.

Now we recall that $\boldsymbol\phi = \jump{\gamma\mathbf U^*}$ and therefore
\[
\|\boldsymbol\varPhi\|_{1/2,\Gamma}^2 = \|\jump{\gamma\mathbf U^*}\|_{1/2,\Gamma}^2 \leq C \|\mathbf U^*\|_{1,\mathbb R^3\setminus\Gamma}^2\leq \frac{C}{\underline{\sigma}^2}\triple{\mathbf U^*}_{|s|,\mathbb R^3\setminus\Gamma}^2.
\]
Analogously, since $\boldsymbol\Lambda = \jump{\mathbf T\mathbf U^*}$, it follows from Lemma \ref{lem:NormalTraceBound}  that
\begin{align}
\nonumber
\|\boldsymbol\Lambda\|_{-1/2,\Gamma}^2 = \|\jump{\mathbf T\mathbf U^*}\|_{-1/2,\Gamma}^2 \leq\,& C_{\Gamma} \frac{|s|}{\underline{\sigma}} \triple{\mathbf U^*}_{|s|,\mathbb R^3\setminus\Gamma}^2.
\end{align}

Using the last two inequalities together with \eqref{eq:est1} we can conclude that
\begin{align*}
\triple{\mathbf U^-}_{|s|,\Omega_-}^2\!\!\! + \|\boldsymbol\varPhi\|_{1/2,\Gamma}^2  + \|\boldsymbol\Lambda\|_{-1/2,\Gamma}^2 \leq  &C \frac{|s|}{\underline{\sigma}^3}\left(\!\triple{\mathbf U^-}_{|s|,\Omega_-}^2\!\!\! + \triple{\mathbf U^*}_{|s|,\mathbb R^3\setminus\Gamma}^2\!\right) \\
\leq  & C \frac{|s|^5}{\sigma^2\underline{\sigma}^7}\left(\|\mathbf F\|_{\Omega_-}^2\!\! + \|\mathbf  T^+ \mathbf U^{inc}\|^2_{-1/2} + \|\gamma^+ \mathbf U^{inc}\|^2_{1/2}\right)\!.
\end{align*}
From which \eqref{eq:est2} follows readily.
\end{proof}
\subsection{Variational formulation of the alternative problem}
We start by defining the operators
\begin{align}
\nonumber
\mathrm A_{\Omega_-}(s) : \boldsymbol H^1(\Omega_-) &\longrightarrow \left(\boldsymbol H^1(\Omega_-)\right)^\prime \\
\nonumber
\mathbf U &\longmapsto (\boldsymbol\sigma(\mathbf U),\boldsymbol\epsilon(\mathbf V))_{\Omega_-}\!\! + s^2(\rho_-\mathbf U^-,\mathbf V)_{\Omega_-}, \\[1ex]
\nonumber
\mathrm B(s) : \boldsymbol H^{-1/2}(\Gamma)\times\boldsymbol H^{1/2}(\Gamma) &\longrightarrow \left(\boldsymbol H^{-1/2}(\Gamma)\times\boldsymbol H^{1/2}(\Gamma)\right)^\prime \\
\nonumber
(\boldsymbol\Lambda,\boldsymbol\varPhi) &\longmapsto \left\langle\begin{pmatrix}\mathcal V(s) & -s^{-1}\left(\tfrac{1}{2}\mathcal I + \mathcal K(s)\right)\\
\left(\tfrac{1}{2}\mathcal I + \mathcal K(s)\right)^\prime & s^{-1}\mathcal W(s) \end{pmatrix} \begin{pmatrix} \boldsymbol\Lambda \\ 
\boldsymbol\varPhi\end{pmatrix}, \begin{pmatrix} \boldsymbol\mu \\ 
\boldsymbol\eta \end{pmatrix}
\right\rangle_\Gamma \\[1ex]
\nonumber
\mathrm B_0(s) : \boldsymbol H^{-1/2}(\Gamma)\times\boldsymbol H^{1/2}(\Gamma) &\longrightarrow \left(\boldsymbol H^{-1/2}(\Gamma)\times\boldsymbol H^{1/2}(\Gamma)\right)^\prime \\
\label{eq:OperatorB0}
(\boldsymbol\Lambda,\boldsymbol\varPhi) &\longmapsto \left\langle\begin{pmatrix} s\mathcal V(s) & -\mathcal K(s)\\ 
\mathcal K(s)^\prime & s^{-1}\mathcal W(s) \end{pmatrix} \begin{pmatrix} \boldsymbol\Lambda \\ 
\boldsymbol\varPhi\end{pmatrix}, \begin{pmatrix} \boldsymbol\mu \\ 
\boldsymbol\eta \end{pmatrix}
\right\rangle_\Gamma \\[1ex]
\nonumber
\mathcal A(s) : \boldsymbol H^1(\Omega_-)\times\boldsymbol H^{-1/2}(\Gamma)\times\boldsymbol H^{1/2}(\Gamma) &\longrightarrow \left(\boldsymbol H^1(\Omega_-)\times\boldsymbol H^{-1/2}(\Gamma)\times\boldsymbol H^{1/2}(\Gamma)\right)^\prime \\
\nonumber
(\mathbf U,\boldsymbol\Lambda,\boldsymbol\varPhi) &\longmapsto \left\langle \begin{pmatrix}\mathrm A_{\Omega_-}(s) & -[(\gamma^-)^\prime ,\, 0\,] \\ & \\ [\gamma^-, 0]^\top & \mathrm B(s)\end{pmatrix} \begin{pmatrix}\mathbf U \\ \boldsymbol\Lambda \\ \boldsymbol\varPhi\end{pmatrix},\begin{pmatrix}\mathbf V \\ \boldsymbol\eta \\ \boldsymbol\mu\end{pmatrix} \right\rangle_\Gamma
\end{align}
where $(\mathbf V,\boldsymbol\mu,\boldsymbol\eta)\in \boldsymbol H^1(\Omega_-)\times\boldsymbol H^{-1/2}(\Gamma)\times\boldsymbol H^{1/2}(\Gamma)$ are arbitrary test functions.

With this notation we can write the variational formulation of \eqref{eq:StrongNonLocal2} as the problem of finding $(\mathbf U^-,\boldsymbol\Lambda,\boldsymbol\varPhi)\in \boldsymbol H^1(\Omega_-)\times\boldsymbol H^{-1/2}(\Gamma)\times\boldsymbol H^{1/2}(\Gamma)$, such that for every $(\mathbf V,\boldsymbol\mu,\boldsymbol\eta)\in \boldsymbol H^1(\Omega_-)\times\boldsymbol H^{-1/2}(\Gamma)\times\boldsymbol H^{1/2}(\Gamma)$
\begin{equation}\label{eq:Variational Form2}
\langle\mathcal A(\mathbf U^-,\boldsymbol\Lambda,\boldsymbol\varPhi),(\mathbf V,\boldsymbol\mu,\boldsymbol\eta)\rangle = \langle \mathbf d_1,\mathbf d_2,\mathbf d_3),(\mathbf V,\boldsymbol\mu,\boldsymbol\eta) \rangle, 
\end{equation}
where $\langle \cdot, \cdot\rangle$ denotes the duality paring between \[
\boldsymbol H^1(\Omega_-)\times\boldsymbol H^{-1/2}(\Gamma)\times\boldsymbol H^{1/2}(\Gamma) \quad \text{ and } \quad \left(\boldsymbol H^1(\Omega_-)\times\boldsymbol H^{-1/2}(\Gamma)\times\boldsymbol H^{1/2}(\Gamma)\right)^\prime,
\]
and the data is given by
\[
(\mathbf d_1,\mathbf d_2,\mathbf d_3) = (\mathbf F + (\gamma^-)^\prime \mathbf T^+\mathbf U^{inc},\gamma^+\mathbf U^{inc},\boldsymbol 0)\in \left(\boldsymbol H^1(\Omega_-)\times\boldsymbol H^{-1/2}(\Gamma)\times\boldsymbol H^{1/2}(\Gamma)\right)^\prime.
\]
To prove the well-posedness of this problem, we will need first prove the following 
\begin{lemma}\label{lem:ElliticityB0}
The operator $\mathrm B_0(s)$ defined in \eqref{eq:OperatorB0} is strongly elliptic in the sense that there exists $C>0$ such that
\begin{equation}
\label{eq:ElliticityB0}
\|\boldsymbol\varPhi\|_{1/2,\Gamma}^2 + 
 \|\boldsymbol\Lambda\|_{-1/2,\Gamma}^2 \leq C\frac{|s|^2}{\sigma\underline{\sigma}^2} \mathrm{Re}\,\left[\langle\mathrm B_0(s)(\boldsymbol\Lambda,\boldsymbol\varPhi),\overline{(\boldsymbol\Lambda,\boldsymbol\varPhi)}\rangle\right].
\end{equation}
\end{lemma}
\begin{proof}
Consider the equation \eqref{eq:AllSpace} that we repeat here for convenience
\[
-\Delta^*_+\mathbf U +s^2\rho_+\mathbf U = \boldsymbol 0 \qquad \text{ in } \mathbb R^3\setminus \Gamma,
\]
and define
\[
\mathbf U = \mathcal S(s)\boldsymbol\Lambda -s^{-1}\mathcal D(s)\boldsymbol\varPhi \qquad \text{ in } \mathbb R^3\setminus\Gamma.
\]
From its definition, it follows that $\mathbf U$ satisfies
\begin{alignat}{8}
\label{eq:LambdaAndPhi}
\jump{\gamma\mathbf U} : =\,& s^{-1}\boldsymbol\varPhi \qquad\qquad&& \jump{\mathbf T\mathbf U} : =\,&& \boldsymbol\Lambda \\
\nonumber
\ave{\gamma\mathbf U} : = \,&  \mathcal V(s)\boldsymbol \Lambda - s^{-1}\mathcal K(s)\boldsymbol\varPhi \qquad\qquad&& \ave{\mathbf T\mathbf U} : = \,&&  \mathcal K^\prime(s)\boldsymbol \Lambda + s^{-1}\mathcal W(s)\boldsymbol\varPhi.
\end{alignat}
Using these identities we can then compute
\begin{align}
\nonumber
\mathrm{Re}\,\left[\langle\mathrm B_0(s)(\boldsymbol\Lambda,\boldsymbol\varPhi),\overline{(\boldsymbol\Lambda,\boldsymbol\varPhi)}\rangle\right] =\,& \mathrm{Re}\,\left[\langle (s\ave{\gamma\mathbf U},\ave{\mathbf T\mathbf U} ),\overline{(\jump{\mathbf T\mathbf U},s\jump{\gamma\mathbf U})}\rangle\right] &\\
\nonumber
=\,& \mathrm{Re}\,\left[\langle s\ave{\gamma\mathbf U},\overline{\jump{\mathbf T\mathbf U}}\rangle_\Gamma + \langle \ave{\mathbf T\mathbf U} ,\overline{s\jump{\gamma\mathbf U})}\rangle_\Gamma\right] &\\
\nonumber
=\,& \mathrm{Re}\,\left[\langle s\gamma^-\mathbf U,\overline{\mathbf T^-\mathbf U}\rangle_\Gamma - \langle s\gamma^+\mathbf U,\overline{\mathbf T^+\mathbf U}\rangle_\Gamma\right] &\\
\nonumber
=\,& \mathrm{Re}\,\left[s(\mathbf U,\overline{\Delta^*\mathbf U})_{\mathbb R^3\setminus\Gamma}+s(\boldsymbol\varepsilon(\mathbf U),\overline{\boldsymbol\sigma(\mathbf U)})_{\mathbb R^3\setminus\Gamma}\right] \qquad& \footnotesize{\text{(From \eqref{eq:BettisFormula})}}\\
\nonumber
=\,&\mathrm{Re}\,\left[\bar{s}(s\rho_+\mathbf U,\overline{s\mathbf U})_{\mathbb R^3\setminus\Gamma}+ s(\boldsymbol\varepsilon(\mathbf U),\overline{\boldsymbol\sigma(\mathbf U)})_{\mathbb R^3\setminus\Gamma}\right] \qquad& \footnotesize{\text{(From \eqref{eq:AllSpace})}}\\
\label{eq:EnergyEquality}
=\,& \sigma \triple{\mathbf U}_{|s|,\mathbb R^3\setminus\Gamma}^2. &
\end{align}

We then use this result combined with \eqref{eq:LambdaAndPhi}, the trace theorem, and \eqref{eq:EnergyEquivalence}  as follows
\[
\|\boldsymbol\varPhi\|_{1/2}^2 = \|s\jump{\gamma\mathbf U}\|_{1/2}^2 \leq c|s|^2\|\mathbf U\|_{\mathbb R^3\setminus\Gamma}^2 \leq c\alpha_0 \frac{|s|^2}{\underline{\sigma}^2}\triple{\mathbf U}_{|s|,\mathbb R^3\setminus\Gamma}^2 = c\alpha_0 \frac{|s|^2}{\sigma\underline{\sigma}^2}\mathrm{Re}\left[\langle\mathrm B_0(s)(\boldsymbol\Lambda,\boldsymbol\varPhi),\overline{(\boldsymbol\Lambda,\boldsymbol\varPhi)}\rangle\right].
\]
Analogously, we use the estimate from Lemma \ref{lem:NormalTraceBound}, the identity \eqref{eq:EnergyEquality}, and the fact that if $s\in\mathbb C_+$ it follows that $1\leq |s|/\underline{\sigma}$ to bound
\[
\|\boldsymbol\Lambda\|_{-1/2}^2 = \|\jump{\mathbf T\mathbf U}\|_{-1/2}^2 \leq c \frac{|s|}{\underline{\sigma}} \triple{\mathbf U}_{|s|,\mathbb R^3\setminus\Gamma}^2\leq c \frac{|s|^2}{\sigma\underline{\sigma}^2} \mathrm{Re}\left[\langle\mathrm B_0(s)(\boldsymbol\Lambda,\boldsymbol\varPhi),\overline{(\boldsymbol\Lambda,\boldsymbol\varPhi)}\rangle\right].
\]
From the two sequences of inequalities above we then conclude that
\[
\|\boldsymbol\Lambda\|_{-1/2}^2 + \|\boldsymbol\varPhi\|_{1/2}^2 \leq C\frac{|s|^2}{\sigma\underline{\sigma}^2} \mathrm{Re}\left[\langle\mathrm B_0(s)(\boldsymbol\Lambda,\boldsymbol\varPhi),\overline{(\boldsymbol\Lambda,\boldsymbol\varPhi)}\rangle\right],
\]
which concludes the proof.
\end{proof}

With the aid of this result, it is possible then to show that the operator $\mathcal A$ from the variational problem \eqref{eq:Variational Form2} is strongly elliptic.

\begin{theorem}
The problem \eqref{eq:Variational Form2} is uniquely solvable. Moreover, there exist constants $C_1>0$ and $C_2>0$ such that
{\small \begin{align}
\label{eq:LStabilityBound}
\left(\triple{\mathbf U}_{|s|,\Omega_-}^2 + \|\boldsymbol\Lambda\|_{-1/2}^2 + \|\boldsymbol\varPhi\|_{1/2}^2\right)^{1/2} \leq&\, C_1 \frac{|s|^{3}}{\sigma\underline{\sigma}^3} \left(\|\mathbf F\|_{\Omega_-}^2 + \|\mathbf T^+\mathbf U^{inc}\|_{-1/2}^2 + \|\gamma^+\mathbf U^{inc}\|_{1/2}^2\right)^{1/2}, \\
\label{eq:LUStabilityBound}
\left(\triple{\mathbf U}_{|s|,\Omega_-}^2 + \triple{\mathbf U^*}_{|s|,\mathbb R^3\setminus\Gamma}^2\right)^{1/2} \leq&\, C_2\frac{|s|^4}{\sigma^2\underline{\sigma}^4} \left(\|\mathbf F\|_{\Omega_-}^2 + \|\mathbf T^+\mathbf U^{inc}\|_{-1/2}^2 + \|\gamma^+\mathbf U^{inc}\|_{1/2}^2\right)^{1/2}.
\end{align} }
\end{theorem}

\begin{proof}
We start by computing
\begin{align}
\nonumber
\mathrm{Re}\Bigg[\Bigg\langle\begin{pmatrix}\overline{s} & 0 & 0 \\ 0 & s & 0\\ 0 & 0 & 1\end{pmatrix}&\mathcal A(s)\begin{pmatrix}\mathbf U \\ \boldsymbol \Lambda \\ \boldsymbol\varPhi\end{pmatrix}, \begin{pmatrix}\overline{\mathbf U} \\ \overline{\boldsymbol \Lambda} \\ \overline{\boldsymbol\varPhi}\end{pmatrix}\Bigg\rangle_\Gamma\Bigg] \\
\nonumber
=&\, \mathrm{Re}\left[\overline{s}(\mathrm A_{\Omega_-}(s)\mathbf U,\overline{\mathbf U})_{\Omega_-}-\overline{s} \langle\boldsymbol\Lambda,\gamma^-\overline{\mathbf U}\rangle_\Gamma + s\langle\gamma^-\mathbf U,\overline{\boldsymbol\Lambda}\rangle_\Gamma  + \left\langle\mathrm B(s)(\boldsymbol\Lambda,\boldsymbol\varPhi),\overline{(\boldsymbol\Lambda,\boldsymbol\varPhi)} \right\rangle_\Gamma\right] \\
\nonumber
=&\, \mathrm{Re}\left[\overline{s}(\mathrm A_{\Omega_-}(s)\mathbf U,\overline{\mathbf U})_{\Omega_-} + \left\langle\mathrm B_0(s)(\boldsymbol\Lambda,\boldsymbol\varPhi),\overline{(\boldsymbol\Lambda,\boldsymbol\varPhi)} \right\rangle_\Gamma \right] \\
\nonumber
=&\, \sigma\triple{\mathbf U}_{|s|,\Omega_-}^2 +  \mathrm{Re}\left[\left\langle\mathrm B_0(s)(\boldsymbol\Lambda,\boldsymbol\varPhi),\overline{(\boldsymbol\Lambda,\boldsymbol\varPhi)} \right\rangle_\Gamma \right] \\
\label{eq:aux100}
\geq\,& \sigma\triple{\mathbf U}_{|s|,\Omega_-}^2 + \frac{\sigma\underline{\sigma}^2}{|s|^2}\left(\|\boldsymbol\Lambda\|_{-1/2}^2 + \|\boldsymbol\varPhi\|_{1/2}^2\right)\\
\label{eq:aux101}
\geq\,& \frac{\sigma\underline{\sigma}^2}{|s|^2}\left(\triple{\mathbf U}_{1,\Omega_-}^2 + \|\boldsymbol\Lambda\|_{-1/2}^2 + \|\boldsymbol\varPhi\|_{1/2}^2\right), 
\end{align}
which implies the unique solvability of \eqref{eq:Variational Form2}. 

We now start from the inequality \eqref{eq:aux100} and use the fact that $(\mathbf U,\boldsymbol\Lambda,\boldsymbol\varPhi)$ satisfies \eqref{eq:Variational Form2} to show that
\begin{align*}
\nonumber
\frac{\sigma\underline{\sigma}^2}{|s|^2}\left(\triple{\mathbf U}_{|s|,\Omega_-}^2 + \|\boldsymbol\Lambda\|_{-1/2}^2 + \|\boldsymbol\varPhi\|_{1/2}^2\right) \leq\,& \mathrm{Re}\Bigg[\Bigg\langle\begin{pmatrix}\overline{s} & 0 & 0 \\ 0 & s & 0\\ 0 & 0 & 1\end{pmatrix}\mathcal A(s)\begin{pmatrix}\mathbf U \\ \boldsymbol \Lambda \\ \boldsymbol\varPhi\end{pmatrix}, \begin{pmatrix}\overline{\mathbf U} \\ \overline{\boldsymbol \Lambda} \\ \overline{\boldsymbol\varPhi}\end{pmatrix}\Bigg\rangle_\Gamma\Bigg] \\
\nonumber
=\,& \mathrm{Re} \Bigg[ \Bigg\langle\begin{pmatrix}\overline{s} & 0 & 0 \\ 0 & s & 0\\ 0 & 0 & 1\end{pmatrix}\begin{pmatrix}\mathbf F + (\gamma^-)^\prime \mathbf T^+\mathbf U^{inc} \\ \gamma^+\mathbf U^{inc} \\ \boldsymbol 0 \end{pmatrix}, \begin{pmatrix}\overline{\mathbf U} \\ \overline{\boldsymbol \Lambda} \\ \overline{\boldsymbol\varPhi}\end{pmatrix}\Bigg\rangle_\Gamma\Bigg]\\
\leq\,& |s| |(\mathbf F,\overline{\mathbf U})_{\Omega_-}+\langle\mathbf T^+\mathbf U^{inc},\overline{\boldsymbol\varPhi}\rangle_\Gamma + \langle\overline{\boldsymbol\Lambda} ,\gamma^+\mathbf U^{inc}\rangle_\Gamma|.
\end{align*}
From this inequality and recalling that $(\mathbf F,\overline{\mathbf U})_{\Omega_-} \leq \underline{\sigma}^{-1}\|\mathbf F\|_{\Omega_-}\triple{\mathbf U}_{|s|,\Omega_-}$ we see that 
\eqref{eq:LStabilityBound} follows readily.

To derive the stability bound \eqref{eq:LUStabilityBound} we first note that the mapping properties of the single and double layer potentials in Lemma \ref{lem:ContinuityBounds} imply that there exists $C>0$ such that
\begin{align*}
 \triple{\mathbf U^*}_{|s|,\mathbb R^3\setminus\Gamma} = \|\mathcal S(s)\boldsymbol\Lambda  - s^{-1}\mathcal D(s)\boldsymbol\varPhi\|_{|s|,\mathbb R^3\setminus\Gamma} \leq\,& C \left(\frac{|s|}{\sigma\underline{\sigma}}\|\boldsymbol\Lambda\|_{-1/2} + \frac{|s|^{1/2}}{\sigma\underline{\sigma}^{1/2}}\|\boldsymbol\varPhi\|_{1/2}\right) \\ 
\leq\,& C \frac{|s|}{\sigma\underline{\sigma}}\left(\|\boldsymbol\Lambda\|_{-1/2} + \|\boldsymbol\varPhi\|_{1/2}\right).
\end{align*}

Hence, it follows that
\[
\triple{\mathbf U}_{|s|,\Omega_-} \!\!+\frac{\sigma\underline{\sigma}}{|s|} \triple{\mathbf U^*}_{|s|,\mathbb R^3\setminus\Gamma} \leq \triple{\mathbf U}_{|s|,\Omega_-} \!\!+ \|\boldsymbol\Lambda\|_{-1/2} + \|\boldsymbol\varPhi\|_{1/2}.
\]
Therefore using \eqref{eq:LStabilityBound} we can see that
\begin{align*}
\frac{\sigma\underline{\sigma}}{|s|}\left(\triple{\mathbf U}_{|s|,\Omega_-} \!\!+\triple{\mathbf U^*}_{|s|,\mathbb R^3\setminus\Gamma}\right) \leq\,& \triple{\mathbf U}_{|s|,\Omega_-} \!\!+\frac{\sigma\underline{\sigma}}{|s|} \triple{\mathbf U^*}_{|s|,\mathbb R^3\setminus\Gamma}\\
\leq\,& \triple{\mathbf U}_{|s|,\Omega_-} \!\!+ \|\boldsymbol\Lambda\|_{-1/2} + \|\boldsymbol\varPhi\|_{1/2} \\
\lesssim \,& \frac{|s|^{3}}{\sigma\underline{\sigma}^3} \left(\|\mathbf F\|_{\Omega_-} \!\!+ \|\mathbf T^+\mathbf U^{inc}\|_{-1/2} + \|\gamma^+\mathbf U^{inc}\|_{1/2}\right),
\end{align*}
from which \eqref{eq:LUStabilityBound} follows.
\end{proof}
\subsection{Contrasting the two formulations}\label{sec:Contrasting}
In the previous sections we have introduced two different formulations for problem \eqref{eq:LaplaceSystem} and derived stability estimates for each of them. A simple comparison of the bounds obtained for both approaches leads to the following conclusion: For the volume unknowns, $\mathbf U^-$ and $\mathbf U^*$, the direct approach using the representation \eqref{eq:IntegralRepresentationA} yields a better result in terms of smaller powers of $|s|$ and $\sigma$ as shown in \eqref{eq:est1}, namely
\[
\left(\triple{\mathbf U^-}_{|s|,\Omega_-}^2 + \triple{\mathbf U^*}_{|s|,\mathbb R^3\setminus\Gamma}^2\right)^{1/2} \leq C_1 \frac{|s|^2}{ \sigma\underline{\sigma}^2}\left(\|\mathbf F\|_{\Omega_-}^2\!\! + \|\mathbf T^+  \mathbf U^{inc}\|^2_{-1/2} + \|\gamma^+ \mathbf U^{inc}\|^2_{1/2}\right)^{1/2}.
\]
However, as we shall see in Section \ref{sec:TDexistence}, when one considers the boundary unknowns  $\boldsymbol\Lambda$ and $\boldsymbol\varPhi$, together with the solution on the interior domain $\mathbf U^-$, 
the bound \eqref{eq:LStabilityBound} stemming from the alternative approach using the representation \eqref{eq:IntegralRepresentationB} 
\[
\left(\triple{\mathbf U^-}_{|s|,\Omega_-}^2 + \|\boldsymbol\Lambda\|_{-1/2}^2 + \|\boldsymbol\varPhi\|_{1/2}^2\right)^{1/2} \leq C_1 \frac{|s|^{3}}{\sigma\underline{\sigma}^3} \left(\|\mathbf F\|_{\Omega_-}^2 + \|\mathbf T^+\mathbf U^{inc}\|_{-1/2}^2 + \|\gamma^+\mathbf U^{inc}\|_{1/2}^2\right)^{1/2},
\]
will translate into better estimates on the growth of the solution for longer times. Hence, in order to obtain optimal information from both boundary and volume unknowns, both approaches are valuable.
%
\section{Results in the time domain}\label{sec:TimeDomain}
\subsection{Preliminaries in the time domain}
Thanks to a result originally due Christian Lubich \cite{Lubich:1994} and to a subsequent refinement by Francisco--Javier Sayas \cite{Sayas:2016}, the bounds obtained in the previous section in terms of the Laplace parameter $s$ and its real part $\sigma$ can then be readily transferred to results regarding the growth of the solution in time as well as the regularity requirements for problem data in the time domain. 

Before introducing the aforementioned theorem, some definitions are required:

\begin{itemize}
\item The space of bounded linear operators between the Banach spaces $X$ and $Y$ will be denoted as $\mathcal B(X,Y)$.
\item The space of continuous functions mapping a real parameter to a Banach space $X$ and having  $k$ continuous derivatives will be denoted as $\mathcal C^k(\mathbb R, X)$.
\item For an analytic function $A(s)$ mapping $\mathbb C_+$ to $\mathcal B(X,Y)$, and $\alpha\in [0,1)$, we will say that $A(\cdot)$ belongs to the class $\mathcal{A} (k + \alpha, \mathcal{B}(X, Y))$ if there exists a non-negative integer $k$ such that 
\[
\|A(s)\|_{X,Y} \le C_A\left(\sigma\right) |s|^{k  + \alpha} \quad \mbox{for}\quad s \in \mathbb{C}_+ ,
\]
where $\sigma:=\mathrm{Re} (s)$, and $C_A : (0, \infty) \rightarrow (0, \infty) $ is a non-increasing function such that 
\[
C_A(\sigma) \le \frac{ c}{\sigma^m} , \quad \forall \quad \sigma \in ( 0, 1]
\]
for some $m \ge 0$ and constant $c>0$. 
\item For a Banach space $X$, the set of $X$-valued time domain distributions will be denoted as $\mathrm{TD}(X)$.

\item We will define the distributional convolution of the $\mathcal B(X,Y)$-valued time-domain distribution $a := \mathcal L^{-1}\{A\}\in \mathrm{TD}(\mathcal B(X,Y)) $ and the  $X$-valued \textit{causal} distribution $g: =\mathcal L^{-1}\{G\}: \in \mathrm{TD}(X)$ as
\[
(a*g)(t) : = \mathcal{L}^{-1}\{AG\}.
\]
\end{itemize}
%
\subsection{Existence, uniqueness and time regularity}\label{sec:TDexistence}

We start by defining the spaces
\begin{align*}
\boldsymbol X :=\,& \boldsymbol L^2(\Omega_-)\times\boldsymbol H^{-1/2}(\Gamma)\times\boldsymbol H^{1/2}(\Gamma), \\
\boldsymbol Y_1 :=\,&  \boldsymbol H^1(\Omega_-)\times \boldsymbol H^1(\mathbb R^3\setminus\Gamma), \\
\boldsymbol Y_2 :=\,& \boldsymbol H^1(\Omega_-)\times \boldsymbol H^{-1/2}(\Gamma)\times \boldsymbol H^{1/2}(\Gamma),
\end{align*}
and the solution operators 
\begin{alignat*}{6}
\mathrm S_{\Omega_-,\mathbb R^3\setminus\Gamma}:\,&& \boldsymbol X \qquad\qquad&\longrightarrow&&\; \quad \boldsymbol Y_1\\
&& (\mathbf F,\mathbf T^+\mathbf U^{inc},\gamma^+\mathbf U^{inc}) & \longmapsto &&\; (\mathbf U^-,\mathbf U^*), \\[2mm]
\mathrm S_{\Omega_-,\Gamma}:\,&& \boldsymbol X \qquad\qquad &\longrightarrow&&\; \quad \boldsymbol Y_2\\
&& (\mathbf F,\mathbf T^+\mathbf U^{inc},\gamma^+\mathbf U^{inc}) & \longmapsto &&\; (\mathbf U^-,\boldsymbol\Lambda,\boldsymbol\varPhi),
\end{alignat*}
where $(\mathbf U^-,\mathbf U^*)$ satisfies \eqref{eq:LaplaceSystem} (with $\mathbf U^+$ replaced by -$\mathbf U^*$) and $(\mathbf U^-,\boldsymbol\Lambda,\boldsymbol\varPhi)$ satisfies \eqref{eq:StrongNonLocal2}.

With all the previous definitions in place, we can now state the result that will allow us to extract time domain results from the Laplace domain results that we have obtained so far. The reader is referred to \cite{Sayas:2016errata}  and \cite[Proposition 3.2.2]{Sayas:2016} for the proof of this result and for further details about Banach space valued time domain and causal distributions.

\begin{theorem}\label{pr:5.1}
Let $A = \mathcal{L}\{a\} \in \mathcal{A} (k + \alpha, \mathcal{B}(X, Y))$ with $\alpha\in [0, 1)$ and $k$ a non-negative integer.  If $ g \in \mathcal{C}^{k+1}(\mathbb{R}, X)$ is causal and its derivative $g^{(k+2)}$ is integrable, then $a* g \in \mathcal{C}(\mathbb{R}, Y)$ is causal and 
\[
\| (a*g)(t) \|_Y \le 2^{\alpha} C_{\epsilon} (t) C_A (t^{-1}) \int_0^t \|(\mathcal{P}_2g^{(k)})(\tau) \|_X \; d\tau,
\]
where 
\[C_{\epsilon} (t) := \frac{1}{2\sqrt{\pi}} \frac{\Gamma(\epsilon/2)}{\Gamma\left( (\epsilon+1)/2 \right) } \frac{t^{\epsilon}}{(1+ t)^{\epsilon}}, \qquad (\epsilon :=  1- \alpha)
\]
the function $\Gamma(\cdot)$ is the Gamma function, and 
\[
(\mathcal{P}_2g) (t) =  g + 2\dot{g} + \ddot{g}.
\]
\end{theorem}

  The Laplace-domain existence and uniqueness results established in Section \ref{sec:LaplaceDomainResults} combined with the invertibility of the Laplace transform imply the existence and uniqueness of the solutions to the time-domain problem. Moreover, the stability bounds obtained for the solutions to \eqref{eq:LaplaceSystem} (keeping in mind the observation made in  \ref{sec:Contrasting}) allow us then to conclude that the solution operators are such that
\begin{align*}
\mathrm S_{\Omega_-,\mathbb R^3\setminus\Gamma} \in\,& \mathcal A(2,\mathcal B(\boldsymbol X,\boldsymbol Y_1)) \quad \text{ with }\quad \alpha =0 \quad \text{ and }\quad C_{\mathcal A}(\sigma) \propto \sigma^{-1}\underline{\sigma}^{-2-1}, \\
\mathrm S_{\Omega_-,\Gamma} \in\,& \mathcal A(3,\mathcal B(\boldsymbol X,\boldsymbol Y_2)) \quad \text{ with }\quad \alpha =0 \quad \text{ and } \quad C_{\mathcal A}(\sigma) \propto \sigma^{-1}\underline{\sigma}^{-3-1}.
\end{align*}

As a consequence of Theorem \ref{pr:5.1}, these observations constitute the proof of the following result regarding the time regularity required from problem data, the time-regularity of the volume and boundary unknowns and their growth in time.

\begin{theorem}\phantom{w}

\begin{enumerate}
\item If $g:=(\mathbf f,\mathbf T^+\mathbf u^{inc},\gamma^+\mathbf u^{inc}) \in \mathcal C^3(\mathbb R,\boldsymbol X)$ is causal and the time derivative $g^{(4)}$ is integrable then $(\mathbf u^-,\mathbf u^*)\in\mathcal C(\mathbb R,\boldsymbol Y_1)$ is causal and
\[
\|\mathbf u^-(t)\|_{\Omega_-} + \|\mathbf u^*(t)\|_{\mathbb R^3\setminus\Gamma} \lesssim \frac{t^2\max\{1,t^3\}}{1+t}\int_0^t \|(\mathcal{P}_2g^{(2)})(\tau) \|_{\boldsymbol X} \; d\tau.
\] 
\item If $g:=(\mathbf f,\mathbf T^+\mathbf u^{inc},\gamma^+\mathbf u^{inc}) \in \mathcal C^4(\mathbb R,\boldsymbol X)$ is causal and the time derivative $g^{(5)}$ is integrable then $(\mathbf u^-,\boldsymbol\Lambda,\boldsymbol\varPhi)\in\mathcal C(\mathbb R,\boldsymbol Y_2)$ is causal and
\[
\|\mathbf u^-(t)\|_{\Omega_-} + \|\boldsymbol\lambda(t)\|_{-1/2} + \|\boldsymbol\varphi(t)\|_{1/2} \lesssim \frac{t^2\max\{1,t^4\}}{1+t}\int_0^t \|(\mathcal{P}_2g^{(3)})(\tau) \|_{\boldsymbol X} \; d\tau.
\]
\end{enumerate}
\end{theorem}

\textbf{A remark on time semidiscretization:} Beyond the information that this result provides about the continuous problem, the theorem also provides a glimpse into the behavior of the discretization error of any Convolution Quadrature based time semidiscretization. In that context, the linearity of the problem implies that the discretization errors satisfy a system with the same structure as the unknowns themselves. The error equations can then be analyzed analogously to what we have done here and a results much like the theorem above can be proven. In that case, the quantities on the left hand sides of the estimates above would correspond to the time-discretization error of each of the variables involved, while the integral on the right hand side is proportional to the accuracy of the time-stepping method used for Convolution Quadrature. The factor involving the time growth would remain identical, thus providing an upper bound to the growth of the time discretization error. The detailed analysis of the fully discretized problem will be the subject of a separate communication.

\section*{Acknowledgements}
Tonatiuh S\'anchez-Vizuet was partially fundeed by the United States National Science Foundation through the grant NSF-DMS-2137305.

\noindent A personal note from George C. Hsiao: \textit{In memory of my sister Laura}.

\newpage
{\small
\bibliographystyle{abbrv} 
\bibliography{references.bib,ghsiao.bib}
}
\end{document}